\documentclass[11pt]{amsart}
\usepackage{bm}
\usepackage{amssymb, adjustbox, enumerate, amsbsy, stmaryrd}
\usepackage{amsmath}
\usepackage[mathscr]{eucal}
\usepackage{amsthm}
\numberwithin{equation}{section}

\usepackage{geometry}

\usepackage[symbol]{footmisc}

\usepackage{amsfonts, amssymb, amscd}

\usepackage{todonotes}

\usepackage{MnSymbol}
\usepackage{bm}
\usepackage{verbatim}
\usepackage{mathrsfs}
\usepackage{graphicx}
\usepackage[all]{xy}
\usepackage{tikz-cd}
\usepackage{subcaption}
\usepackage{listings}
\usepackage{subfiles}
\usepackage[toc,page]{appendix}
\usepackage{mathtools}
\usepackage{comment}
\usepackage{enumerate}
\usepackage{enumitem}
\usepackage[linesnumbered,ruled]{algorithm2e}

\usepackage{graphicx}
\graphicspath{{images/}}

\usepackage{appendix}
\usepackage{hyperref}
\hypersetup{
	colorlinks=true,
	citecolor=red,
	linkcolor=blue,
	filecolor=magenta,      
	urlcolor=red,
}
\newcommand{\Rr}{\mathbb{R}}
\newcommand{\RR}{\mathbb{R}}

\newcommand{\Nn}{\mathbb{N}}

\newcommand{\pP}{\mathcal{P}}

\newcommand{\kK}{\mathcal{K}}

\newcommand{\Ll}{\mathcal{L}}

\newcommand{\Ss}{\mathcal{S}}

\newcommand{\Vv}{\mathcal{V}}

\theoremstyle{definition}

\newtheorem{theorem}{Theorem}[section]
\newtheorem{lemma}[theorem]{Lemma}
\newtheorem{proposition}[theorem]{Proposition}
\newtheorem{corollary}[theorem]{Corollary}

\theoremstyle{definition}

\newtheorem{definition}[theorem]{Definition}

\newtheorem{remark}[theorem]{Remark}

\begin{document}
	
	\title[$C^\infty$ free boundary for fractional Laplacian]{$C^{\infty}$ Regularity for the free boundary of one-phase Fractional Laplacian problem}

	\begin{abstract}
		We consider a one-phase free  boundary  problem involving fractional Laplacian $(-\Delta)^s$, $0<s<1$. D. De Silva, O. Savin, and Y. Sire proved that the flat boundaries are $C^{1,\alpha}$. We raise the regularity to $C^{\infty}$, extending the result known for $(-\Delta)^{1/2}$ by D. De Silva and O. Savin.
	\end{abstract}

\author[R. ~Lyu]{Runcao Lyu}
\thanks{lyuruncao@u.nus.edu}
	\maketitle
	\tableofcontents
	\section{Introduction}
	Since the 1970s, free boundary problems have become one of the most important subjects in the field of elliptic PDE. One of the most famous problems is the classical \textit{one-phase} Alt-Caffarelli or Bernoulli problem analyzed by H. Alt and L. Caffarelli in \cite{Caffarelli1981}. It is
	\begin{equation}
		\left\{\begin{aligned}
			&u\geq0,&\quad&\text{ in }\Omega,\\
			&\Delta u=0,&\quad&\text{ in }\Omega\cap\{u>0\},\\
			&|Du|=1,&\quad&\text{ on }\Omega\cap\partial\{u>0\}.
		\end{aligned}\right.
	\end{equation}
	It is a natural model for long-range interactions appearing in large systems of interacting particles, flame propagation, and probability theory. Those models raised the interest of mathematicians to analyze the regularity of the free boundary of $u$ \cite{tangentball,FernandezReal2024,Duvaut1976}. Specifically, mathematicians aim to find conditions under which the free boundary is smooth. A series of papers \cite{Caffarelli1981,DeSilva2019,Li2023,Kinderlehrer1977} proved that locally flat free boundaries are $C^\infty$. On the other hand, \cite{Caffarelli1981,Weiss1998} proved that the set of singular points on the free boundary is small in the sense of Hausdorff dimension. We refer to \cite{Caffarelli2005,FernandezReal2022,Velichkov2023} for further details in this field.  We also refer to some important papers on the regularity of the solutions \cite{RosOton2016,RosOton2017,Figalli2020a,Sire2021a,Sire2021,Fall2022,FernandezReal2024a,FernandezReal2024b} for some Schauder-type regularity results.  Note that the regularity of the free boundary and that of the solution are closely related.
	
	In this paper, we focus on the one-phase \textit{fractional Laplacian}  problem. The fractional Laplacian operator is a standard model for \textit{non-local} elliptic operators. Over the past two decades, non-local free boundary problems, especially those involving the fractional Laplacian operator, have attracted significant attention  \cite{Caffarelli2007a,tangentball,Caffarelli2013,Caffarelli2017,Barrios2018,Barrios2018a,Abatangelo2020,RosOton2025,RosOton2025c}.

	The fractional Laplacian operator is a Fourier multiplier in $\Rr^n$ with symbol is $|\xi|^{2s}$. Alternatively, we can define the operator  by an integral representation:
	\[(-\Delta)^su(x)=c_{n,s}P.V.\int_{\Rr^n}\frac{u(x)-u(x')}{|x-x'|^{n+2s}}dx',\quad0<s<1,\]
	where $c_{n,s}$ is a constant depending on $n$ and $s$, and $P.V.$ denotes the Cauchy principal value. In this paper, we consider the one-phase Alt-Caffarelli problem for the fractional Laplacian  $(-\Delta)^s$, $0<s<1$, in an open domain 
	$\Omega\subset\Rr^n$, following \cite{tangentball,DeSilva2014,Engelstein2021},
	\begin{equation}\label{fractional laplacian}
		\left\{\begin{aligned}
			&u_{0}\geq0,\quad&\text{ in }&\Omega,\\
			&(-\Delta)^s u_{0}=0,\;s\in(0,1),\quad&\text{ in }&\Omega\cap\{u>0\},\\
			&\lim_{t\to 0^+}\frac{u_{0}(x+t\nu(x))}{t^s}=1,\quad&\text{ on }&F(u):=\Omega\cap\partial\{u>0\},
		\end{aligned}\right.
	\end{equation}
	where $u_0$ is defined on the entire space $\Rr^n$ with prescribed values outside $\Omega$. In particular, the case $s=\frac{1}{2}$ was studied in greater detail in \cite{DeSilva2012,DeSilva2012b,DeSilva2012a,Silva2015}.

	To study the problem, we will use the Caffarelli-Silvestre extension. It plays a key role in our work. As noted earlier, the non-local nature of the fractional Laplacian precludes the direct application of classical local PDE techniques to the function $u_0$. However, the work of Caffarelli and Silvestre in \cite{Caffarelli2007,Caffarelli2016} allows us to extend $u_{0}$ from $\Rr^{n}$ to $\Rr^{n+1}$. Consider the following minimizing problem
	\begin{equation}
		\min \left\{\int_0^\infty\int_{\Rr^{n}}y^\beta|\nabla u(x,y)|^2dxdy:\;\beta=1-2s,u\Big|_{\Rr^n\times\{0\}}=u_{0}\right\}.
	\end{equation}
	The minimizer $u$ satisfies the following equation:
	\begin{equation}
		\left\{\begin{aligned}
			&L_\beta u=0,\quad\text{in }\Rr^n\times(0,+\infty),\\
			&u=u_{0},\;\;\quad\text{on }\Rr^n\times\{0\},\\
			&(-\Delta)^s u_{0}(x)=d_{n,s}\lim_{y\to 0^+}\frac{u(x,y)-u(x,0)}{y^{1-\beta}},
		\end{aligned}\right.
	\end{equation}
	where $d_{n,s}$ is a constant depending on $n$ and $s$, and $L_\beta$ is defined by \[L_\beta(u):=\text{div}(|y|^\beta\nabla u).\]
	The operator $L_\beta$ is local, allowing us to use standard elliptic techniques. This extension transforms the non-local problem for $u_0$ into a local one for $u$. Consequently, it suffices to study the following boundary value problem:
	\begin{equation}\label{extension of fracional laplacian}
		\left\{
		\begin{aligned}
			&u\geq 0,\quad &\text{in}&\; B_1,\\
			&L_\beta u = 0, \quad &\text{in}& \; B_1 \cap\{u>0\}, \\
			&\frac{\partial u}{\partial \nu^s} = 1, \quad &\text{on}& \; F(u),
		\end{aligned}
		\right.
	\end{equation}
	where \(-1<\beta = 1 - 2s<1\), $\nu$ is the interior normal unit vector of $F(u)$, and
	\begin{equation}
		\frac{\partial u}{\partial \nu^s}(x_0) := \lim_{t \to 0^+} \frac{u(x_0 + t\nu(x_0))}{t^s}, \quad x_0 \in F(u).
	\end{equation}

	For general $0<s<1$, the optimal regularity of solutions has been established in \cite{tangentball,RosOton2024}. The $C^{1,\alpha}$ regularity for flat boundaries was established in \cite{DeSilva2012,DeSilva2014,RosOton2025b}.	The smallness of the singular set was proven in \cite{Engelstein2021}. Specifically, \cite{Silva2015} improved the regularity of flat free boundary for $s=\frac{1}{2}$ to $C^\infty$.
	
	As noted in \cite[Remark 1.2]{RosOton2025b}, the best-known regularity result for the free boundary when $s\neq \frac{1}{2}$ is $C^{1,\alpha}$. A recent paper \cite{Barrios2025} proved that $C^{2,\alpha}$ free boundaries are $C^\infty$ for general non-local operators. We state the $C^{1,\alpha}$ regularity result by D. De Silva, O. Savin, and Y. Sire. Our work is based on that regularity result.

	\begin{theorem}(\cite[Theorem 1.1]{DeSilva2014})\label{C1alpha}
		There exists a small constant $\bar{\varepsilon} > 0$ depending on $n$ and $\beta$, such that if $u$ is a viscosity solution to $(\ref{extension of fracional laplacian})$ satisfying
		\begin{equation}
			\{x \in B_1 : x_n \leq -\bar{\varepsilon} \} \subset \{x \in B_1 : u(x, 0) = 0\} \subset \{x \in B_1 : x_n \leq \bar{\varepsilon} \},
		\end{equation}
		then $F(u)$ is $C^{1,\alpha}$ in $B_{1/2}$. Here $\alpha$ depends on $n,\beta$.
	\end{theorem} 
	
	\vspace{5pt}
	
	In this paper, we improve the regularity of flat free boundary \textbf{from $C^{1,\alpha}$ to $C^\infty$}. It is our main Theorem.
	\begin{theorem}\label{fractional cinfty}
		There exists a small constant $\bar{\varepsilon} > 0$ depending on $n$ and $s$, such that if $u_{0}$ is a viscosity solution to $(\ref{fractional laplacian})$ satisfying
		\begin{equation}
			\{x \in B_1 : x_n \leq -\bar{\varepsilon} \} \subset \{x \in B_1 : u_{0}(x, 0) = 0\} \subset \{x \in B_1 : x_n \leq \bar{\varepsilon} \},
		\end{equation}
		then the free boundary $F(u_0)$ is $C^{\infty}$ in $B_{1/2}$.
	\end{theorem}
	Using Caffarelli-Silvestre extension, the fractional Laplacian can be extended to $\Rr^{n+1}$. Throughout this paper, we mainly consider the extended problem. We will prove the extended version of Theorem \ref{fractional cinfty}:
	\begin{theorem}\label{C infty of free boundary}
		There exists a small constant $\bar{\varepsilon} > 0$ depending on $n$ and $\beta$, such that if $u$ is a viscosity solution to $(\ref{extension of fracional laplacian})$ satisfying
		\begin{equation}
			\{x \in B_1 : x_n \leq -\bar{\varepsilon} \} \subset \{x \in B_1 : u(x, 0) = 0\} \subset \{x \in B_1 : x_n \leq \bar{\varepsilon} \},
		\end{equation}
		then $F(u)$ is $C^{\infty}$ in $B_{1/2}$.
	\end{theorem}
	\begin{remark}
		Our strategy for proving the smoothness of $C^{1,\alpha}$ free boundaries consists of two steps.  In step 1, we prove that $C^{1,\alpha}$ free boundaries are $C^{2,\alpha}$.	In step 2, we prove that $C^{2,\alpha}$ free boundaries are $C^\infty$. Step 2  is also recently proved in the paper \cite{Barrios2025} by B. Barrios, X. Ros-Oton, and Marvin Weidner by considering a general type of global operators. In this paper, we prove by using the Caffarelli-Silvestre extension. This approach is significantly different from the proof in \cite{Barrios2025}. Our approach is inspired by the work of De Silva and Savin \cite{DeSilva2012b,Silva2015}. The techniques in \cite{Jhaveri2017,Sire2021a,Sire2021,Fall2022,FernandezReal2024} by M. Fall, X. Fernández-Real, Y. Jhaveri, R. Neumayer, X. Ros-Oton, Y. Sire, S. Terracini, and S. Vita provided key insights into the regularity theory.
		
		Now we introduce our strategy in detail. In step 1, we prove following following the spirit of \cite{DeSilva2012b}. We use a sequence of paraboloids to approximate the free boundary. For each paraboloid, we use function $V_{\Ss,a,b}$ to approximate the solution $u$. During the approximation, we need to analyze the difference of the approximating function  $V_{\Ss,a,b}$  and the solution $u$. For $\beta=0$,  the limit is still a harmonic function with natural regularity estimate. For general $\beta$, we use Schauder-type estimates to derive analogous results.
		
		In step 2, we prove following following the spirit of \cite{Silva2015}. The function $\frac{u_i}{u_n}$ is an  extension of the derivatives of $g_i$ in \eqref{standard Fu}. We prove that it has the same regularity as $g$. The desired result will follow. For general $\beta\in (-1,1)$, we first analyze the linearized problem \eqref{linearized problem} and find polynomial approximation for the solution. That needs us to analyze the degenerate operator  $L_\beta(U_n\cdot)$. The result together with the compactness allows to appoximate  the function $\frac{u_i}{u_n}$.
		
		We notice that step 1: $C^{1,\alpha}\to C^{2,\alpha}$ is essential. The technique in step 2 cannot be applied to $C^{1,\alpha}$ boundaries. To use the technique, we need to analyze the linearized problem with a Neumann boundary condition. However, as \cite{Silva2015} indicates, the Neumann condition cannot be verified for $C^{1,\alpha}$ boundary.

		We follow the spirit of \cite{DeSilva2012b,Silva2015} to construct our proof. Compared to the proof for $\beta=0$, there are several difficulties. For $\beta=0$, the operator $L_\beta=\Delta$. It has been fully analyzed. For general $\beta\in (-1,1)$, we use the regularity theory about $L_\beta$. Roughly speaking, we will consider the regularity theory for the operator $\Delta+\frac{\beta}{y}\partial_y$ since $L_\beta=|y|^\beta(\Delta+\frac{\beta}{y}\partial_y)$. Secondly, harmonic functions remain harmonic under holomorphic map. Therefore to analyze a solution $w$ to $\Delta (U_nw)=0$ in $\Rr^n\setminus\pP$, we can consider the slice of the space $\{x'=0\}$ and map the 2D slit domain to the upper half plane with a holomorphic map. For general $\beta$, we will use boundary Harnack inequality in slit domain to handle it.
		
	\end{remark}
	
	\vspace{7pt}
	
	The paper is organized as following. In Section \ref{preliminaries}, we recall some preliminary results and introduce some notations. In Sections \ref{a family of functions} to \ref{C2alpha regularity}, we raise the regularity to $C^{2,\alpha}$. In Section \ref{a family of functions}, we introduce a family of functions such that their free boundary are paraboloids. In Section \ref{C2alpha regularity}, we use the functions to approximate the solution $u$ to prove the $C^{2,\alpha}$ regularity.  In Section \ref{cinfty}, we prove $C^\infty$ regularity by analyzing the linearized problem to show that the free boundary $F(u)$ has the same regularity as its derivatives. In Section \ref{expansion theorem section}, we consider the linearized problem with $F(u)$ being a codimension 2 hyperplane. In Section \ref{appendix}, we provide some known results about this problem.

	The regularity of flat free boundaries plays a key role in understanding Bernstein-type problems. We are preparing a follow-up paper based on \cite{Engelstein2023} to analyze graphical solutions. This will help us find an equivalent formulation for Bernstein-type problems.
	
	\noindent\textbf{Acknowledgments}: The author would like to thank Hui Yu and Ovidiu Savin for fruitful discussions regarding this project. We refer to a recent paper \cite{Barrios2025} by B. Barrios, X. Ros-Oton, and Marvin Weidner. It proved that $C^{2,\alpha}$ free boundaries are smooth differently. They also developed the smoothness for obstacle problem.
	\section{Preliminary}\label{preliminaries}
	In this section, we recall some definitions and gather some results about the fractional Laplacian problem. First, we introduce the notation used in this paper.
	\subsection{Notation}
	In this paper, we consider the problem \eqref{extension of fracional laplacian} in $\Rr^{n+1}$. For $X \in \mathbb{R}^{n+1}$, we use $x'$ and $x$ to denote its first $n-1$ and $n$ coordinates. We use $x_n$,$y$ to denote the $n$-th, $(n+1)$-th coordinates and write $X=(x,y)=(x',x_n,y)$.
	
	In $\mathbb{R}^{n+1}$, a ball with radius $r$ and center $X$ is denoted by $B_r(X)$. For simplicity we use $B_r$ to denote $B_r(0)$. We use $B_r^n$ to denote the $n$-dimensional ball $B_r \cap \{y = 0\}$.
	
	The half-plane $\pP$ and codimension-two space $\Ll$ will frequently appear. We define them
	\begin{equation}\label{half plane P}
		\pP=\{(x,y):x_n\leq 0,y=0\},\quad\Ll=\{(x,y):x_n=y=0\}.
	\end{equation}
	
	For a continuous non-negative function $u:\Rr^{n+1}\to\Rr_+$, we define its free boundary $F(u)$ as \[F(u)= \partial_{\mathbb{R}^n} (\Rr^n\times\{0\}\cap \{u>0\})\cap B_1 \subset \mathbb{R}^n\times \{0\},\]Specifically, we focus on  $F(u)$ satisfying:
	\begin{equation}\label{standard Fu}
		F(u)=\{(x',g(x'),0):g(0)=0,g\in C^{1,\alpha},\nabla g(0)=0\},
	\end{equation} where $e_n$ is the unit vector  in the $n$-th direction.
	
	Throughout this paper, $n,\beta$ are called universal constants. We say a variable is universal if it only depends on universal constants. We say a variable has universal bound if it is bounded by a universal constant.
	
		In this paper, we consider a different version of regularity following \cite{Silva2015}. Using the notation\[
	x^{\mu} = x_1^{\mu_1} \cdots x_n^{\mu_n}, \quad |\mu| = \mu_1 + \cdots + \mu_n, \quad \mu_i \geq 0,\quad\|P\| := \max |a_{\mu m}|.
	\] We denote by
	\[
	P(x, r) = a_{\mu m} x^{\mu} r^m, \quad \deg P = k,
	\]
	a polynomial of degree \(k\) in the \((x, r)\) variables. 
	\subsection{Viscosity solution}
	Before considering viscosity solutions, we introduce the function $U:\Rr^{2}\to \Rr$. Set
	\begin{equation}\label{half space solution}
		U(t, y) = (\frac{-t+\sqrt{t^2+y^2}}{2})^{s}.
	\end{equation}
	For $t=x_n$,  $U\in C^2(\Rr^{n+1}\setminus\pP)$, and $U(x_n,y)$ is a classical solution to \eqref{extension of fracional laplacian}. Furthermore, the free boundary $F(U)=\Ll$ and $\{U=0\}=\pP$. Therefore $U$ is denoted by \textit{half-plane solution}. We use $U_t$ to denote $\frac{\partial U}{\partial t}$. The following equalities will be of use in this paper.
	\begin{equation}\label{derivatives of U}
		U_t=s\frac{U}{r},\;U_{tt}=\frac{sr-t}{r^2}U_t=\frac{s(sr-t)}{r^3}U,\;r=\sqrt{t^2+y^2}.
	\end{equation}
		
	Now we construct the definition of viscosity solutions. It allows us to consider a solution to be less regular. To begin with, we state the definitions of touching functions following \cite{DeSilva2014}.
	
	\begin{definition} 
		For continuous functions  $u$ and $v$, we say $v$ \textit{touches $u$ from below} (resp. above) at $X_0 \in B_1$ if $u(X_0) = v(X_0)$ and  
		\[  
		u(X) \geq v(X) \quad (\text{resp.} \quad u(X) \leq v(X)) \quad \text{in a neighborhood } O \text{ of } X_0.  
		\]  
		If this inequality is strict in $O \setminus \{X_0\}$, we say that $v$ touches $u$ strictly by below (resp. above).  
	\end{definition}  
	With the concept of touching functions, we give the definition of comparison sub/sup-solutions.
	\begin{definition}
		A function $v \in C(B_1)$ is a \textit{comparison subsolution} to (\ref{extension of fracional laplacian}) if $v$ is non-negative in $B_1$, even with respect to $\{y = 0\}$,  $C^2$ in the set $\{v>0\}$, and satisfies  
		
		(i)\; $L_\beta v \geq 0 \quad \text{in} \quad B_1 \setminus \{y = 0\},$
		
		(ii) $F(v)$ is $C^2$ and if $x_0 \in F(v)$ we have  
		\[  
		v(x, y) = aU((x - x_0) \cdot \nu(x_0), y) + o(|(x - x_0, y)|^s), \quad \text{as} \quad (x, y) \to (x_0, 0),  
		\]  
		with $a \geq 1$, where $\nu(x_0)$ denotes the unit normal at $x_0$ to $F(v)$ pointing toward $\{v>0\}$;
		
		Moreover, $v$ is a strict comparison subsolution if either (i) holds strictly or $a > 1$.  Similarly, one can define a (strict) comparison supersolution.
	\end{definition} 
	Finally, we are able to give the definition of viscosity solutions to end this subsection  following \cite{Caffarelli1995,DeSilva2014}. The idea of the definition is that solutions cannot be touched from below/above by sub/sup-solutions. It needs less regularity of $F(u)$ and uniform convergence implies that the limit is close to a solution. Throughout this paper, all solutions are viscosity solutions unless otherwise specified.
	\begin{definition}
		A function $u$ is a \textit{viscosity solution} to \eqref{extension of fracional laplacian} if $u$ is continuous, non-negative in $B_1$, and satisfies:  
		
		(i) $u$ is locally $C^{1,1}$ in $B_1\cap \{u>0\}$, even with respect to $\{y = 0\}$ and solves (in the viscosity sense)  
		\[  
		L_\beta u = 0 \quad \text{in} \quad B_1 \setminus \{y = 0\}.
		\]  
		
		(ii) Any (strict) comparison subsolution (resp. supersolution) cannot touch $g$ from below (resp. by above) at a point $X_0 = (x_0, 0) \in F(u)$.
		
	\end{definition}

	For this equation, we need to consider a different type of regularity. Now we give the definition of $f \in C^{k,\alpha}_{xr}$. For usually regularity, we use polynomials of $(x.y)$ to approximate the function. Using $P(x,r)$ to approximate the function induces $C^{k,\alpha}_{xr}$ regularity.
	\begin{definition}
		For $F(u)$ satisfying \eqref{standard Fu}, we say that a function \( f : B_1 \subset \mathbb{R}^{n+1} \to \mathbb{R} \) is pointwise \( C^{k,\alpha} \)
		in the \((x,r)\)-variables at \( 0 \in F(u) \), and write \( f \in C^{k,\alpha}_{xr}(0) \), if there exists a (tangent)
		polynomial \( P_0(x,r) \) of degree \( k \) such that
		\[
		f(X) = P_0(x,r) + O(|X|^{k+\alpha}).
		\]
		We define \( \|f\|_{C^{k,\alpha}_{xr}(0)} \) as the smallest constant \( M \) such that
		\[
		\|P_0\| \leq M \quad \text{and} \quad |f(X) - P_0(x,r)| \leq M |X|^{k+\alpha},
		\]
		for all \( X \) in the domain of definition.
		Furthermore,  we say that \( f \in C^{k,\alpha}_{xr}(K) \) for \( K \subset F(u) \) if there exists a constant
		\( M \) such that \( f \in C^{k,\alpha}_{xr}(Z) \) for all \( Z \in K \), and
		\[
		\|f\|_{C^{k,\alpha}_{xr}(Z)} \leq M \quad \text{for all } Z \in K.
		\]
		The smallest such constant \( M \) is denoted by \( \|f\|_{C^{k,\alpha}_{xr}(K)} \).
	\end{definition}

	\subsection{The Linearized Problem}  
	In this subsection, we introduce here the linearized problem associated to (\ref{extension of fracional laplacian}). To improve the regularity of the free boundary $F(u)$, we will show that $\frac{u_i}{u_n}$, which extends the derivatives of $F(u)$, has the same regularity as $F(u)$. They are solutions to the linearized problem.
	
	For $u$ being a solution to \eqref{extension of fracional laplacian} satisfying \eqref{standard Fu}, the linearized problem associated to (\ref{extension of fracional laplacian}) is 
	\begin{equation}\label{general linearized problem}
		\left\{\begin{aligned}
			&L_\beta(u_n w) = 0, \quad &\text{in }& B_1\cap\{u>0\},\\
			&|\nabla_r w| = 0, \quad &\text{on }& B_1\cap F(u). 
		\end{aligned}\right.
	\end{equation}
	Here $|\nabla_r w|$ denotes the radial derivative, defined as\[  
	|\nabla_r w|(X_0) := \lim_{\substack{(t,y) \to (0,0) \\ r^2 = t^2 + y^2}} \frac{w(X_0+t\nu+ye_{n+1}) - w(X_0)}{r},
	\]
	for $\nu$ being the normal unit of $F(u)$ at $X_0$.
	
	A special case of \eqref{general linearized problem} is when $u$ is half-plane solution and $F(u)=\Ll$. In this case, $u$ is substituted by the half-plane solution $U$. To deal with that, we need to consider the viscoisity type solution. We introduce the following problems following the spirit of \cite{DeSilva2014}.
	\begin{definition}\label{definition of solution to linearized problem}
		
		We say $w$ is a viscosity solution to\begin{equation}\label{linearized problem}
			\left\{\begin{aligned}
				&L_\beta(U_n w) = 0, \quad &\text{in }& B_1\setminus\pP,\\
				&|\nabla_r w| = 0, \quad &\text{on }& B_1\cap\Ll,
			\end{aligned}\right.
		\end{equation} if $w\in C^{1,1}_{loc}(B_1\setminus\pP)$, $w$ is even with respect to $\{y=0\}$, and $w$ satisfies:
		
		(i) $L_\beta(U_n w) = 0$ in $B_1 \setminus \{y = 0\}$;  
		
		(ii) Let $\phi$ be a continuous function around $X_0 = (c,0,0) \in B_1\cap\Ll$ satisfying
		\[  
		\phi(X) = \phi(X_0) + a(X_0) \cdot (x - c) + b(X_0)r + O(|x - c|^2 + r^{1+\alpha}), 
		\] for some $\alpha>0$. If $b(X_0) > 0$, then $\phi$ cannot touch $w$ from below at $X_0$, and if $b(X_0) < 0$ then $\phi$ cannot touch $w$ from above at $X_0$.  
	\end{definition}
	The solution to \eqref{linearized problem} has a polynomial expansion at points in $B_1\cap\Ll$. We have the following theorem allowing use to use polynomials to approximate the solution following the  spirit of \cite[Theorem 6.1]{DeSilva2014}. 
	\begin{theorem}\label{quadratic expansion}
		Given a boundary data $\bar{h}\in C(\partial B_1)$ even with respect to $\{y=0\}$, $|\bar{h}|\leq 1$, there exists a unique classical solution $h$ to \eqref{linearized problem} such that $h\in C(\bar{B}_1)$ is even with respect to $\{y=0\}$ and satisfies $h=\bar{h}$ on $\partial B_1$. For any $ k\in \Nn$, there exist $Q$ being order $k$ polynomial in $x'$  and $P$ being  order $k-1$ polynomial in $x,r$  such that
		\[|h-(Q(x')+rP(x,r))|\leq C_o|X|^{k+1}\text{ in }B_{1/2},\]
		for $C_o$ depending on universal constants and $k$. Furthermore, $\|Q\|,\|P\|$ have universal bounds and $L_\beta(U_n(Q+rP))=0$ in $B_1\setminus\pP$. 
		
		In particular, for $n=2$ we have
		\begin{equation}\label{expansion inequlity}
			\Big|h(X)-\Big(h(0)+\xi_0\cdot x'+\frac{1}{2}{x'}^TM_0x'-\frac{1}{2s}\Big(\frac{a_0}{2}r^2+b_0rx_n\Big)\Big)\Big|\leq C_o|X|^3\text{ in }B_1\setminus\pP,
		\end{equation}
		for some $a_0,,b_0\in \RR,\xi_0\in\RR^{n-1},M_0\in S^{(n-1)\times(n-1)}$ with
		\[|\xi_0|,|a_0|,|b_0|,\|M_0\|,C_o\leq C(n,\beta),\]
		and
		\[\frac{a_0}{2s}+b_0-tr M=0.\]
	\end{theorem}

	\section{A Family of Functions}\label{a family of functions}
	
	In this section, we introduce a family of functions $V_{\Ss,a,b}$ following the spirit of \cite{DeSilva2012b} by D. De Silva and O. Savin. They serve as approximate solutions to  \eqref{extension of fracional laplacian} with free boundary $\Ss$. We mainly study the case $\Ss$ being a paraboloid.

	Given a surface $\Ss=\{x_n=g(x')\}\subset\mathbb{R}^n$. For any $X=(x,y)\in\mathbb{R}^{n+1}$, we set
	\begin{equation}\label{family of functions}\begin{aligned}
			&U_0(X)=U(t,y),&\quad &t=\text{sgn}(x_n-g(x'))\text{dist}((x,0),\Ss),\\
			&v_{a,b}(t,y):=(1+\frac{a}{4}r+\frac{b}{2}t)U_0(t,y),&\quad &r^2=t^2+y^2,\\
			&V_{\Ss,a,b}(X):=v_{a,b}(t,y),&\quad &X=(x,y).
		\end{aligned}
	\end{equation}
	
	Next we introduce the concept of domain variation functions. It will be used to measure the distance between $v$ and $U$, especially for $v=V_\Ss$ close to $U$. For a function $v$ close to the half-plane solution $U$, we construct the the domain variation function $\tilde{v}$.  In Section \ref{C2alpha regularity} we will construct a sequence of functions approximating the solution $v$. Domain variations will be an important role in the approximating process.
	
	\begin{definition}\label{domain variation definition}
		Let $ \epsilon > 0 $ and let $ v $ be a continuous non-negative function in $ \overline{B_\rho} $. 
		We follow \cite{DeSilva2012b} to define the multivalued map $\tilde{v}_\epsilon:\Rr^{n+1}\setminus \pP\to \Rr$ via the formula 
		\begin{equation} 
			U(X) = v(X - \epsilon w n), \quad \forall w \in \tilde{v}_\epsilon(X).  
		\end{equation}  
		We call $ \tilde{v}_\epsilon $ the $ \epsilon $-\textit{domain variation} associated to $v$. By abuse of notation, from now on we write $ \tilde{v}_\epsilon(X) $ to denote any of the values in this set.
	\end{definition} If 
	\begin{equation} \label{flatness assumption}
		U(X - \epsilon e_n) \leq v(X) \leq U(X + \epsilon  e_n) \quad \text{in } B_\rho\text{ for }\epsilon>0,
	\end{equation}  
	then we have $\tilde{v}_\epsilon(X)\neq \emptyset$ for $X\in B_{\rho-\epsilon}\setminus \pP$ and $|\tilde{v}_\epsilon|\leq 1$. 
	Moreover, if $ v $ is strictly monotone in the $e_n$-direction in $ B_\rho(v)\cap\{v>0\} $, then $ \tilde{v}_\epsilon $ is single-valued.
	
	Domain variations also have the comparison principle. It indicates that domain variation away from the free boundary will control its behavior near the free boundary in some way.
	\begin{lemma}\label{maximum principle}(\cite[Lemma 3.1]{DeSilva2014})
		Let $ u, v $ be respectively a solution and a subsolution to (\ref{extension of fracional laplacian}) in $ B_2 $, with $ v $ strictly increasing in the $ e_n $-direction in $ B_2(v)\cap\{v>0\} $. Assume that $ u $ satisfies the flatness assumption (\ref{flatness assumption}) in $ B_2 $ for $ \epsilon > 0 $ small and that $ \tilde{v}_\epsilon $ is defined in $ B_{2 - \epsilon} \setminus \pP $ and bounded. If  
		\begin{equation}
			\tilde{v}_\epsilon + c \leq \tilde{u}_\epsilon \quad \text{in} \quad (B_{3/2} \setminus \overline{B_{1/2}}) \setminus \pP,  
		\end{equation}  
		then  
		\begin{equation} 
			\tilde{v}_\epsilon + c \leq \tilde{u}_\epsilon \quad \text{in} \quad B_{3/2}\setminus \pP.  
		\end{equation}  
	\end{lemma}

	In fact, the behavior of free boundary can imply the property of the domain variation. It can be understood as following: the regularity of free boundary we derive the regularity of solution near it. For a weaker condition of free boundary, we are supposed to find a similar result for the solution. We have:
	\begin{lemma}\label{compare freeboundary with plane}(\cite[Lemma 2.10]{DeSilva2014})
		Assume $u$ solves (\ref{extension of fracional laplacian}). Given any $\epsilon > 0$ there exists $\delta > 0$ and $\bar{\epsilon} > 0$, depending on $\epsilon$ such that if  
		\[  
		\{(x,0) \in B_1 : x_n \leq -\bar{\epsilon}_0\} \subset \{(x,0) \in B_1 : u(x,0) = 0\} \subset\{(x,0) \in B_1 : x_n \leq \bar{\epsilon}_0\},  
		\]  
		then the rescaling $u$ satisfies  
		\[  
		U(X - \delta\epsilon e_n) \leq u(X) \leq U(X + \delta\epsilon e_n) \quad \text{in } B_\delta.  
		\]   
	\end{lemma}
	
	We use the notation  
	\begin{equation}
		V_{M,\xi',a,b}(X) := V_{\Ss,a,b}(X),  
	\end{equation}for $ \Ss := \{ x_n =\frac{1}{2} (x')^T M x' + \xi'\cdot x' \} $, $M$ being a symmetric $(n-1)\times(n-1)$ matrix, and $\xi'\in \mathbb{R}^{n-1}$. This is the  main case under consideration in this paper.
	Furthermore, we will investigate $V_{M,\xi',a,b}(X)$ with some certain conditions. We give the definition below.
	\begin{definition} \label{definition of Vdelta}
		For $\delta > 0$ small, we define the following classes of functions  
		\begin{align*}  
			\Vv_\delta &:= \{ V_{M,\xi',a,b} : \|\text{M}\|,|\xi'|, |a|, |b| \leq \delta \}, \\  
			\Vv_\delta^0 &:= \{  V_{M,\xi',a,b} \in \Vv_\delta : \frac{a}{2s} +b - \text{tr}M = 0 \}.  
		\end{align*}  
	\end{definition}
	We define the rescaling $V_\lambda$ of $V_{M,\xi',a,b}$ to be  
	\begin{equation*}  
		V_\lambda(X) = \lambda^{-s} V(\lambda X)=V_{\lambda M,\xi',\lambda a, \lambda b}(X), \quad X \in B_1.
	\end{equation*} 
	
	To use $\Vv_\delta$ to approximate solution to \eqref{extension of fracional laplacian}, we need to compare them with solutions. In the following proposition, we provide a condition such that $V\in\Vv_\delta$ is a subsolution/supersolution.
	
	\begin{proposition}\label{subsolution verification}
		Let $V = V_{M,\xi',a,b} \in \Vv_\delta$, with $\delta \leq \delta_0(n,\beta)$. There exists a constant $C_0 > 0$ depending  only on $n,\beta$ such that if  
		\begin{equation}  
			\frac{a}{2s} +  b - \text{tr}M \geq C_0 \delta^2,  
		\end{equation}  
		then $V$ is a comparison subsolution to (\ref{extension of fracional laplacian}) in $B_2$.  
	\end{proposition}
	\begin{proof}
		We can easily verify the boundary condition for $V$. Our main task, therefore, is to show that $L_{\beta}V\geq 0$ in $B_2$. To this end, we now calculate the action of the operator on $V$. We begin by recalling that
		\begin{equation}\label{Laplacian t}
			\Delta_x t(x)=-\kappa(x),\;\text{$\kappa(x)$ being the mean curvature of the parallel surface to $\Ss$ at $x$.}
		\end{equation} For the case $ \Ss := \{ x_n =\frac{1}{2} (x')^T M x' + \xi'\cdot x' \} $, we have
		\begin{equation}\label{mean curvature}
			|\kappa(x)-\text{tr}M|\leq C\delta^2\quad\text{for }\|M\|, |\xi'|\leq \delta\text{ and }x\in B_2.
		\end{equation}
		Applying \eqref{Laplacian t} and \eqref{mean curvature}, we have
		\[\begin{aligned}
			L_\beta V(X)&=|y|^\beta((\frac{a}{2s} +  b)U_t-(\partial_t v_{a,b})\kappa(x)),
		\end{aligned}\]
		for $X=(x,y)\in B_2$. By analysis we have $|\kappa(x)-trM|\leq C\delta^2$ for $C$ depending  only on $n,\beta$. For completeness, we show the calculation below:
		
		\[\begin{aligned}
			(\Delta+\frac{\beta}{y}\partial_y)_{(t,y)}v_{a,b}=(\frac{a}{2s}+b)U_t.
		\end{aligned}\]
		
		Since $r\leq 2$ and $|\xi'|,\|M\|\leq \delta$, we have
		\begin{equation}\label{derivative of v_ab}
			\begin{aligned}
				|(\partial_t v_{a,b}-U_n)(t,y)|\leq C\delta U_t.
			\end{aligned}
		\end{equation}
		Therefore, we obtain the following estimate
		\begin{equation}\label{estimate of Lbeta}
			|L_\beta V-|y|^\beta(\frac{a}{2s} +  b-trM)U_t|\leq \frac{1}{2}C_0|y|^\beta \delta^2U_t,
		\end{equation}
		for $C_0$ universal. If $a+b-trM\geq C_0\delta^2$ we have the desired result.
	\end{proof}
	Next we derive the comparison between $V_n$,$L_\beta V$ and $U_n$ for $\delta$ small. It indicates that $V$ is close to $U$ for $\delta$ small. Furthermore, we have a quantitative estimate for $V_n,L_\beta V$. We will use it for the comparison between $V$ and $u$.
	\begin{proposition}\label{compare Vn Un}
		Let $V=V_{M,\xi',a,b}\in \Vv_\delta$ with $\delta\leq\delta_0$, $\delta_0$ universal. Here $\Vv_\delta$ is defined in Definition \ref{definition of Vdelta}. We have
		\begin{equation}\label{domain to estimate V_n}
			c\leq\frac{V_n}{U_n}\leq C\quad \text{in } B_2\setminus(\pP\cup\{|(x_n,y)|\leq 10\delta|x'|\}),
		\end{equation}
		here $\pP$ is defined in \eqref{half plane P}. If $V\in\Vv_\delta^0$, then
		\begin{equation}\label{Lbeta estimate}
			|(\Delta+\frac{\beta}{y}\partial_y)V(X)|\leq C\delta^2 U_n(X)\quad \text{in } B_1\setminus(\pP\cup\{|(x_n,y)|\leq C'\delta|x'|\}).
		\end{equation}
	\end{proposition}
	\begin{proof}
		As \eqref{Lbeta estimate} is a direct result of \eqref{estimate of Lbeta} and \eqref{domain to estimate V_n}, it suffices to prove \eqref{domain to estimate V_n}. In step 1 we establish the proportionality between $V_n$
		and $U_t(t,y)$, where $t$ defined in \eqref{family of functions}. In step 2 we substitute $t$ by $x_n$.
		
		\vspace{5pt}
		
		\textit{Step 1: Proportionality between $V_n$ and $U_t(t,y)$.}
		
		\vspace{5pt}
		
		By defintion, we have
		\[\partial_t v_{a,b}(t,y)\geq V_n(X)=\partial_t v_{a,b}(t,y)\frac{\partial t}{\partial x_n}\geq \frac{1}{2}\partial_t v_{a,b}(t,y)\text{ for $\delta$ small}.\]
		We apply \eqref{derivative of v_ab} to compare $\partial_t v_{a,b}$ and $ U_t$ to obtain
		\begin{equation}\label{estimate V_n by U_t}
			2U_t(t,y)\geq V_n(X)\geq \frac{1}{4}U_t(t,y).
		\end{equation} 
		
		\vspace{5pt}
		
		\textit{Step 2: Substitute $t$ by $x_n$.}
		
		\vspace{5pt}

		It suffices to derive the comparison between $U_t(t,y)$ and $U_n(x_n,y)$. From Harnack inequality, we have \begin{equation}\label{comparable U_t}
			\frac{U_t(t_1,y)}{U_t(t_2,y)}\leq C\quad\text{if}\quad|t_1-t_2|\leq \frac{1}{2}|(t_2,y)|.
		\end{equation} For the domain given in (\ref{domain to estimate V_n}), we have
		\[|(x_n,y)|\geq 8\delta |x|\geq 2|t-x_n|.\]
		Therefore $B_{|t-x_n|}(x_n,y)\subset B_{2|t-x_n|}(x_n,y)\subset \{U>0\}$. Applying \eqref{comparable U_t} for $t_1=t$ and $t_2=x_n$, the proof is done. 
	\end{proof}
	 The following proposition gives an estimate of $\tilde{V}$ for $V\in \Vv_\delta$. For $\delta$ small, we know that the domain variation is close to $0$. As the Proposition above, the following proposition gives a quantitative estimate for it. This will be an important tool in our paper.
	\begin{proposition}\label{difference between tildeV and gammaV} Let $V = V_{M,\xi',a,b} \in V_\delta$, with $\delta \leq \delta_0$ universal. Then $V$ is strictly monotone increasing in the $e_n$-direction in $B_2^+(V)$. Moreover, $V$ satisfies  
			\begin{equation}\label{gamma definition}
				|\tilde{V}(x) - \gamma_V(x)| \leq C_1 \delta^2 \text{ in }B_2\setminus \pP, \quad \gamma_V(x) = \frac{1}{2s}(\frac{a}{2}r^2+brx_n)-\frac{1}{2}(x')^TMx'-\xi'\cdot x',
			\end{equation}
		$\text{with } r = \sqrt{x_n^2 + y^2} \text{ and } C_1 \text{ depending  only on $n,\beta$}.  $
	\end{proposition}
	
	\begin{proof}  
		The monotonicity of $V$ follows from (\ref{estimate V_n by U_t}). The proof of the inequality is similar to the proof of Proposition \ref{compare Vn Un}. In step 1 we derive the estimate for $t$. It is quite the case in 2D space such that we do not need to consider the terms about $x'$. In step 2 we establish the estimate \eqref{gamma definition} by considering the difference between $t + \gamma_{a,b}(t, s)$ and $x_n + \gamma_V(X)$.

		\vspace{5pt}
		
		\textit{Step 1: Estimate for $t$.}
		
		\vspace{5pt}

		In this step we prove that $v_{a,b}$ satisfies  
		\begin{equation} \label{gamma estimate}
			U(t + \gamma_{a,b} - C\delta^2, y) \leq v_{a,b}(t, y) \leq U(t + \gamma_{a,b} + C\delta^2, y),\quad \gamma_{a,b}(t,y):=\frac{1}{2s}(\frac{a}{2}r^2+br t),
		\end{equation}  
		for $C$ depending  only on $n,\beta$.
		
		To begin with, we apply the Taylor expansion to $U$ and \eqref{derivatives of U} to derive 
		\[|U(t + \mu, y) - (U(t,y) + \mu U_t(t, y))|
		\leq \frac{1}{2}\mu^2 |U_{tt}(t',y)|.\]
		By \eqref{derivatives of U}, we can estimate of $U_t$ and $U_{tt}$ to obtain:
		\[|U(t + \mu, y) - (1+\frac{s\mu}{r})U(t,y) |\leq Cs \mu^2 r^{-2} U(t,y),\quad r=\sqrt{t^2+y^2},\]
		for  some $t'$ satisfying $t\leq t'\leq t+\mu$ and $|\mu| \leq 2/r$.
		
		Choosing $\mu = \tilde{\mu}\pm2\frac{C\tilde{\mu}^2}{sr}$, we obtain 
		\begin{equation}\label{difference tildeV and gammaV step}
			U(t + \tilde{\mu} + \frac{C\tilde{\mu}^2}{sr}, y) \geq (1 + \frac{s\tilde{\mu}}{r}) U(t,y) \geq U(t + \tilde{\mu} - \frac{C\tilde{\mu}^2}{sr}, y),  
		\end{equation}
		provided that $|\tilde{\mu}/r| < c(n,\beta)$ sufficiently small.  Set $\tilde{\mu} = \frac{1}{2s}(\frac{a}{2} r^2 + bt r)$ in \eqref{difference tildeV and gammaV step}, then \eqref{gamma estimate} is obtained.

		\vspace{5pt}
		
		\textit{Step 2: Proof of \eqref{gamma definition}.}
		
		\vspace{5pt}

		In this step, we will prove
		\begin{equation}\label{flatness of V}
			U(X + (\gamma_V(X) - C \delta^2) e_n) \leq V(X) \leq U(X + (\gamma_V(X) + C \delta^2) e_n).
		\end{equation} We notice
		\[\gamma_V(X) = \gamma_{a,b}(x_n, s) - \frac{1}{2} x'^T M x' - \xi' \cdot x'.\]
		Therefore it suffices to prove
		\begin{equation}\label{compare t xn}
			|(t + \gamma_{a,b}(t, s)) - (x_n + \gamma_V(X))| \leq C \delta^2,
		\end{equation}
		then \eqref{flatness of V} follows by applying the  inequality to \eqref{gamma estimate}. We will prove the \eqref{compare t xn} by comparing $t$ and $x_n$. We have
		\[t = 0 \text{ on } \Ss := \{x_n = g(x') := \frac{1}{2} x'^T M x' + \xi' \cdot x'\},\quad 1 \geq \frac{\partial t}{\partial x_n} \geq 1 - C \delta^2 \text{ in } B_1.\]
		Integrating this inequality on the segment $(x', h(x')), (x', x_n)$, we obtain
		\[|t - (x_n - g(x'))| \leq C \delta^2.\]
		
		In $B_1$, the surface $\Ss$ and $x_n = 0$ are within distance $\delta$ from each other. We
		have $|t - x_n| \leq C \delta$. Hence
		\[|\gamma_{a,b}(t, y) - \gamma_{a,b}(x_n, y)| \leq || \nabla \gamma_{a,b} ||_\infty |t - x_n| \leq C \delta^2.\]
		
		The last two inequalities prove \eqref{flatness of V}. The monotonicity of $U$ in the $e_n$ direction gives the estimate for $\tilde{V}$.
	\end{proof} 
	\begin{remark}\label{delta flatness remark}
		Proposition  \ref{difference between tildeV and gammaV} can give an estimate of $V$ by half-plane solution $U$. If $V\in\Vv_{\delta/4}$, then $V$ satisfies\begin{equation}\label{delta flatness}
			U(X-\delta e_n)\leq  V(X)\leq U(X+\delta e_n)\quad\text{in }B_1.
		\end{equation}
		We call \eqref{delta flatness} the $\delta-$flatness assumption in $B_1$.
	\end{remark}
	Next we vaompare the functions $V$. In the  following lemma, we show that the comparison of two surfaces will imply the comparison of two functions.
	\begin{lemma}\label{paraloid separate}
		Let $\Ss_1$ and $\Ss_2$ be surfaces with boundaries of curvature not exceeding 2. Let
		\[V_i=V_{\Ss_i,a_i,b_i},\quad |a_i|,|b_i|\leq 2, \quad i=1,2.\] 
		Assume that
		\[\Ss_i\cap B_{2\sigma}=\{x_n=g_i(x')\},\quad \sigma\leq c,\]
		with $g_i$ Lipschitz graphs,  $g_i(0)=0$, $|\nabla g_i|\leq 1$ and $c$ depending only on $n,\beta$. If
		\[|a_1-a_2|,\;|b_1-b_2|\leq\epsilon,\;\|g_1-g_2\|_{L^\infty(B_\sigma)}\leq \epsilon\sigma^2,\]
		for some small $\epsilon\leq c$, then for some universally big $C$ we have
		\[V_1(X)\leq V_2(X+C\epsilon\sigma^2 e_n)\quad \text{in}\; B_\sigma.\]
	\end{lemma}  
	
	\begin{proof}
		After a rescaling of factor $1/\sigma$, it suffices to prove our lemma for $\sigma = 1$ and with the curvature of $S_i$, $a_i$, $b_i$ and $\epsilon$ smaller than $c$ universal. This proof is similar to the proofs above. In step 1 we prove the desired result for $t$. It is quite the 2D space case. In step 2 we establish the desired result by considering the difference between dist$(X,\Ss_i)$, $i=1,2$.

		\vspace{5pt}
		
		\textit{Step 1: Estimate for $t$.}
		
		\vspace{5pt}

		In this step we prove that for $0 < \lambda \leq 1$, 
		\[
		v_{a_1, b_1}(t, y) \leq v_{a_2, b_2}(t + C\epsilon\lambda^2, y), \quad \lambda \leq r = |(t, y)| \leq 2\lambda.
		\]
		
		By (\ref{estimate V_n by U_t}), $\partial_t v_{a,b}$ is proportional to $\partial_t U$ in the disk of radius 2. Since on the segment with endpoints $(t, y)$ and $(t + C\epsilon\lambda^2, y)$ all the values of $\partial_t U$ are comparable (see (\ref{comparable U_t})) we obtain (using $r U_t =s U$)
		\[\begin{aligned}
			v_{a_2, b_2}(t + C\epsilon\lambda^2, y) \geq v_{a_2, b_2}(t, s) + C \epsilon \lambda^2 U_t(t, y)\geq v_{a_1, b_1}(t, y),
		\end{aligned}\]
		and our claim is proved.
		
		\vspace{5pt}
		
		\textit{Step 2: Estimate for dist$(X,\Ss_i)$, $i=1,2$}
		
		\vspace{5pt}
		
		Since $v_{a_2, b_2}$ is increasing in the first coordinate, we obtain that
		\[
		v_{a_1, b_1}(t, y) \leq v_{a_2, b_2}(t + C\epsilon, y), \quad |(t, y)| \leq 1.
		\]
		On the other hand, from the hypotheses on $h_i$ we see that 
		\[
		t_1 + C\epsilon \leq \bar{t}_2\text{ in } B_1,
		\]
		where $\bar{t}_2$ is the distance to $\mathcal{S}_2 - C' e_{n}$, for some $C'$ large depending on the $C$ above. Hence in $B_1$ we have
		\[
		V_1(X) = v_{a_1, b_1}(t_1, y) \leq v_{a_2, b_2}(t_1 + C\epsilon, y) \leq v_{a_2, b_2}(\bar{t}_2, y) = V_2(X + C' \epsilon e_{n}).
		\]
	\end{proof}
	
	By the preparation above, we now seek for Harnack-type Theorem \ref{g V harnack} following the spirit of \cite{DeSilva2012b}. It will be an important tool in Section \ref{C2alpha regularity}. To begin with, we prove the following lemma. It shows that the comparison of $u$ and $V$ at a given point indicates the comparison near the free boundary.
	
	\begin{lemma}\label{lemma harnack}
		If $u$ solves (\ref{extension of fracional laplacian}) and
		\begin{equation}\label{g bigger than V}
			u(X)\geq V(X-\epsilon e_n)\quad \text{in} \; B_1,
		\end{equation}
		\[u(\bar{X})\geq V(\bar{X})\quad \text{at some}\;\bar{X}\in B_{1/8}(\frac{1}{4}e_n),\]
		with $V=V_{M,\xi',a,b}\in\Vv_\delta^0$, $\bar{C}\delta^2\leq \epsilon$ for $\bar{C}>0$ universal, then
		\[u(X)\geq V(X-(1-\tau)\epsilon e_n)\quad \text{in}\;\;\mathcal{C},\]
		\[\mathcal{C}:=\{(x',x_n,y):\sqrt{s}d\leq |(x_n,y)|\leq\frac{1}{2}\}, \quad d=\frac{1}{8\sqrt{n-1}},\]
		for $\tau\in(0,1)$ universal.
	\end{lemma}
	\begin{proof}
		We consider Taylor expansion of $V$ at $X-\epsilon e_n$. It gives that for any $\tau\in (0,1)$, there exists $\lambda\in (0,1)$ such that \[V(X-(1-\tau)\epsilon e_n)=V(X-\epsilon e_n)+\tau\epsilon V_n(X+\lambda\epsilon e_n).\]By \eqref{domain to estimate V_n} and \eqref{comparable U_t}, $V_n(X+\lambda\epsilon e_n)$ is comparable to $U_n(X)$. We derive
		\begin{equation}\label{difference V}V(X-(1-\tau)\epsilon e_n)\leq  V(X-\epsilon e_n)+C_1\tau\epsilon U_n(X).
		\end{equation}
		
		Set $h(X):=u(X)-V(X-\epsilon e_n)$. It suffices to show
		\begin{equation}\label{difference g V}
			h\geq c_1\epsilon U_n\quad \text{in }\mathcal{C}.
		\end{equation}With \eqref{difference V} and \eqref{difference g V}, we have the desired result by  setting $\tau=c_1/C_1$.
		
		We prove \eqref{difference g V} in three steps. In step 1, we derive the equations satisfied by $h$. In step 2, we prove it for values of $h$ in $B_{1/8}(\frac{1}{4} e_n)$ using Theorem \ref{interior harnack} Harnack inequality, Theorem \ref{interior harnack}. In step 3, we prove it for values out of $B_{1/8}(\frac{1}{4} e_n)$ using Theorem \ref{boundary harnack} boundary Harnack inequality, Theorem \ref{boundary harnack}.
		
		\vspace{5pt}
		
		\textit{Step 1: Equations satisfied by $h$.}
		
		\vspace{5pt}
		
		In this step, we estimate $h(\bar{X})$ and $L_\beta h$. Based on \eqref{g bigger than V}, we have $h\geq 0$ in $B_1$ and
		\begin{equation}\label{eq:4.15}
			h(\bar{X})\geq V(\bar{X})-V(\bar{X}-\epsilon e_n)\geq c\epsilon U_n(\bar{X})\geq c_2\epsilon.
		\end{equation}
		Notice that the same as \eqref{difference V}, the second inequality is derived by \eqref{domain to estimate V_n} and \eqref{comparable U_t}.
		
		Next we estimate $L_\beta h=|y|^\beta(\Delta+\frac{\beta}{y}\partial_y) h$. Based on (\ref{Lbeta estimate}), we have
		\[
		|(\Delta+\frac{\beta}{y}\partial_y) h| \leq C \delta^2 U_n \leq C_2 \delta^2 \quad \text{in } \tilde{\mathcal{C}} \setminus \pP,
		\]
		where $\pP$ is the half plane in \eqref{half plane P} and $\tilde{\mathcal{C}} \supset \supset \mathcal{C}$ is the $\frac{\sqrt{s}d}{2}$-neighborhood of $C$.

		\vspace{5pt}
		
		\textit{Step 2: Values of $h$ in $B_{1/8}(\frac{1}{4} e_n)$.}
		
		\vspace{5pt}
		
		We set $h=h_1+h_2$ satisfying
		\[\left\{\begin{aligned}
			&h_1=h,\quad&\text{on }&\partial (B_{3/16}(\frac{1}{4} e_n)),\\
			&L_\beta h_1=0,\quad&\text{in }&B_{3/16}(\frac{1}{4} e_n),
		\end{aligned}\right.\]and
		\[\left\{\begin{aligned}
			&h_2=0,\quad&\text{on }&\partial (B_{3/16}(\frac{1}{4} e_n)),\\
			&|L_\beta h_2|\leq C|y|^\beta\delta^2,\quad&\text{in }&B_{3/16}(\frac{1}{4} e_n).
		\end{aligned}\right.\]
		Consider $\varphi(X)=-|X-\frac{1}{4}e_n|^2+(\frac{3}{16})^2$. We have $\varphi=0$ on $\partial (B_{3/16}(\frac{1}{4} e_n))$ and $L_\beta \varphi=-2(n+1+\beta)|y|^\beta$. By maximum principle, we have $h_2\leq \frac{C\delta^2}{2(n+1+\beta)}\varphi\leq C\delta^2$. Apply interior Harnack inequality to $h_1$, we have that $h_1\geq c_2\epsilon$ in $B_{1/8}(\frac{1}{4} e_n)$, for some universal $c_2>0$. 
		
		Finally, we estimate $h$ to end this step. For $\bar{C}$ large enough, $\delta^2$ is much smaller than $\epsilon$. Therefore we have:
		\begin{equation} 
			h \geq c_2 \epsilon - C \delta^2 \geq c_3 \epsilon, \quad \text{in } B_{1/8}(\frac{1}{4} e_n).
		\end{equation}
		
		\vspace{5pt}
		
		\textit{Step 3: Values of $h$ out of $B_{1/8}(\frac{1}{4} e_n)$.}
		
		\vspace{5pt}
		
		Now we proceed to estimate $h$  out of $B_{1/8}(\frac{1}{4} e_n)$. In this domain, the values near $\pP$ are more sophicated. We will use Theorem \ref{boundary harnack} boundary Harnack inequality to estimate it. As in step 2, the boundary and $L_\beta h$ effect $h$ in opposite ways. We will use two functions $q_1$ and $q_2$ to estimate them. Set
		\[
		D := \tilde{\mathcal{C}} \setminus \left( B_{1/8}(\frac{1}{4} e_n) \cup \pP \right)
		\]
		where $\pP$ as in \eqref{half plane P} and $\tilde{\mathcal{C}}$ is the $\frac{\sqrt{s}d}{2}$-neighborhood of $\mathcal{C}$. To handle boundary effects, we construct auxiliary functions $q_1,q_2$ satisfying
		\begin{equation} \label{eq:4.16}
			(\Delta+\frac{\beta}{y}\partial_y) q_1 = 0, \quad (\Delta+\frac{\beta}{y}\partial_y) q_2 = -1, \quad \text{in }D,
		\end{equation}
		with boundary conditions respectively
		\[
		q_1 = 0 \quad \text{on } \partial \tilde{\mathcal{C}} \cup \pP, \quad q_1 = 1 \quad \text{on } \partial B_{1/8}(\frac{1}{4} e_n),
		\]
		\[
		q_2 = 0 \quad \text{on } \partial D.
		\]
		
		By Theorem \ref{boundary harnack} boundary Harnack inequality, $q_1$ is comparable to $q_2+\frac{y^2}{2(1+\beta)}$ in a neighborhood of $\pP \cap \mathcal{C} \subset \subset \tilde{\mathcal{C}}$. For the points away from $\pP$, we can simply use interior Harnack to derive $q_1\lesssim q_2$ as in step 2. Then we have
		\begin{equation} \label{eq:4.17}
			q_1 \geq c_4 q_2 \quad \text{in } \mathcal{C} \setminus B_{1/8}(\frac{1}{4} e_n),
		\end{equation}
		with $c_4 > 0$ universal. By the maximum principle,
		\[
		h \geq q := c_3 \epsilon q_1 - C_2 \delta^2 q_2 \quad \text{in } D,
		\]
		since $h \geq q$ on $\partial D$ and $\Delta h \leq \Delta q$ in $D$. Hence, by \eqref{eq:4.17} we have
		\[
		h \geq \frac{c_3}{2} q_1 \geq c_5 \epsilon U_n, \quad \text{in } \mathcal{C} \setminus B_{1/8}(\frac{1}{4} e_n),
		\]
		where in the last inequality we used that $q_1$ and $U_n$ are comparable. We obtain this as a direct result of Theorem \ref{boundary condition of v} boundary Harnack inequality.
	\end{proof}
	With the lemma above, we prove  the following proposition. It somehow indicates that if $F(u)$ is between two paraboloids, then we can improve the approximation in a smaller ball. It is similar to the  improvement of flatness in \cite{Silva2011} by D. De Silva.
	\begin{proposition}\label{prop for Harnack}
		There exist $\bar{\epsilon}, \bar{\delta} > 0$ universally small and $C > 0$ universally big, such that if $u$ solves \eqref{extension of fracional laplacian} and it satisfies
		\[
		V(X - \epsilon e_n) \leq u(X) \leq V(X + \epsilon e_n) \quad \text{in } B_1, \quad \text{for } 0 < \epsilon \leq \bar{\epsilon},
		\]
		with
		\[
		V = V_{M, \xi', a, b} \in \Vv_\delta^0, \quad \delta \leq \bar{\delta}, \quad \bar{C}\delta^2 \leq \epsilon,
		\]
		then either\begin{equation}\label{g leq V}
			u(X) \leq V(X + (1 - \eta) \epsilon e_n) \quad \text{in } B_\eta,
		\end{equation}
		or\begin{equation}\label{g geq V}
			u(X) \geq V(X - (1 - \eta) \epsilon e_n) \quad \text{in } B_\eta,
		\end{equation}
		for a small universal constant $\eta \in (0, 1)$. Here $\Vv_\delta^0$ is defined in Definition \ref{definition of Vdelta}.
	\end{proposition}
	\begin{proof} 
		Without loss of generality, we assume
		\[u(\bar{X})-V(\bar{X})\geq 0,\quad \bar{X}=\frac{1}{2}e_n,\]
		and prove \eqref{g geq V}. For the case $u(\bar{X})-V(\bar{X})\leq 0$, we can prove 
		\eqref{g leq V} in the same way. We compare $\tilde{u}$ and $\tilde{V}$ in this proof rather than considering the comparison of $u$ and $V$ directly. To begin with, we apply Lemma \ref{lemma harnack} to derive:
		\begin{equation}\label{tildeg bigger that tildeV}
			\tilde{u}\geq \tilde{V}+(\tau-1)\epsilon,\quad\text{in }\mathcal{C}'\setminus \pP,
		\end{equation}
		where $\pP$ is defined in \eqref{half plane P} and $
		\mathcal{C}':=\{(x',x_n,y):\sqrt{s}d\leq |(x_n,y)|\leq\frac{1}{4},|x'|\leq\frac{1}{2}\},\;d:=\frac{1}{8\sqrt{n-1}}.$

		 We prove \eqref{tildeg bigger that tildeV} in 2 steps. In step 1, we construct a subsolution $\tilde{W}$ close to $\tilde{u}$. In step 2, we transfer the comparison between $\tilde{u}$ and $\tilde{W}$ to the comparison between $\tilde{u}$ and $\tilde{V}$. 
		
		\vspace{5pt}
		
		\textit{Step 1: Subsolution $W$.}
		
		\vspace{5pt}
		
		In this step, we perturb $V$ to construct $W$, ensuring it remains a subsolution while capturing the error terms. Then we can apply Lemma \ref{maximum principle}. We set
		\[W:=V_{M+\frac{c}{n-1}\epsilon I,\xi',a,b+2c\epsilon}\in\Vv_{\delta+\epsilon},\]
		for $c$ small enough to be chosen. As $\frac{a}{2s}+(2c\epsilon+b)-(\text{tr}(M)+c\epsilon)=c\epsilon\geq C_0(\delta+\epsilon)^2$, $W$ is a subsolution by Proposition \ref{subsolution verification}.
		
		By Proposition \ref{difference between tildeV and gammaV}, we estimate the difference of $\tilde{V}$ and $\tilde{W}$
		\begin{equation}\label{compare tildeV tildeW}
			|\tilde{V}-\tilde{W}+c\epsilon(\frac{rx_n}{s}-\frac{|x'|^2}{2(n-1)})|\leq 2C_1(\delta+\epsilon)^2\quad\text{in }B_1.
		\end{equation}
		
		By choosing $\bar{C}$ big and $c,\bar{\epsilon}$ small such that $4C_1\bar{\epsilon}+4C_1/\bar{C}+2c/s\leq \frac{\tau}{2}$, we have
		\[\tilde{V}\geq \tilde{W}-\frac{\tau}{2}\epsilon,\quad\text{in }B_2.\]
		Similarly, by choosing $\bar{C}$ big such that $\tau^*:=7cd^2-4C_1/\bar{C}-4C_1\epsilon$ satisfying $0<\tau^*<\tau/2$, we have
		\[\tilde{V}\geq \tilde{W}+7c\epsilon d^2-4C_1/\bar{C}\epsilon-4C_1\epsilon^2\geq \tilde{W}+\tau^*\epsilon\quad\text{on }\{|(x_n,y)|\leq \sqrt{s}d,|x'|=\frac{1}{2}\}\setminus \pP.\]
		Together with \eqref{tildeg bigger that tildeV}, we have
		\[\tilde{u}\geq \tilde{W}+(\tau^*-1)\epsilon\quad\text{in }(\mathcal{C}'\cup \{|(x_n,y)|\leq \sqrt{s}d,|x'|=\frac{1}{2}\})\setminus \pP.\]
		We apply Lemma \ref{maximum principle} and conclude
		\begin{equation}\label{compare tildeg tildeW}
			\tilde{u}\geq \tilde{W}+(\tau^*-1)\epsilon\quad\text{in }\{|(x_n,y)|\leq \sqrt{s}d,|x'|\leq\frac{1}{2}\}\setminus \pP.
		\end{equation}

		\vspace{5pt}
		
		\textit{Step 2: Compare $\tilde{V}$ and $\tilde{W}$.}
		
		\vspace{5pt}

		In step 1, we have proven \eqref{compare tildeg tildeW}. It suffices to substitute $\tilde{W}$ by $\tilde{V}$. From \eqref{compare tildeV tildeW}, we have that  
		\[\tilde{W}\geq \tilde{V}-\frac{\tau^*}{2}\epsilon\text{ in }B_{2\eta}\setminus\pP,\]
		$\text{if }c\frac{(2\eta)^2}{2(n-1)}+4C_1/\bar{C}+4C_1\bar{\epsilon}\leq \frac{\tau^*}{2}\text{ and } 2\eta<\sqrt{s}d.$ The inequality is satisfied by choosing $\bar{C}$ big and $\bar{\epsilon}$ small.
		Combine the above result with \eqref{compare tildeg tildeW}, we have
		\[\tilde{u}\geq \tilde{W}+(\tau^*-1)\epsilon\geq\tilde{V}+(\tau^*/2-1)\epsilon\geq\tilde{V}+(\eta-1)\epsilon\quad\text{in }B_{2\eta}\setminus \pP.\]
		
		Finally, we prove the desired result. From \eqref{flatness of V} and $V$ being monotone, $u,$ $V$ are $(4\delta+\epsilon)$-flat  in $B_1$ as in Remark \ref{delta flatness remark}. For any $X\in B_\eta$, we set $\exists Y=X+\tilde{u}(Y) e_n$. We require $\eta\geq 4\bar{\delta}+\bar{\epsilon}$ such that $Y\in B_{2\eta}$. The desired result follows since
		\[u(X)=U(Y)=V(Y-\tilde{V}(Y)e_n)\geq V(Y-\tilde{u}(Y)e_n+(\eta-1)\epsilon e_n)=V(X+(\eta-1)\epsilon e_n).\]
	\end{proof}
	With the proposition above, we naturally wish to raise the approximation inductively. We cannot raise the approximation infinitely because that each time we raise it, the surface is farther from plane. The following theorem describes how much we can raise the approximation in a given ball.
	\begin{theorem}\label{g V harnack}
		There exist $\bar{\epsilon} > 0$ small and $\bar{C} > 0$ large universal, such that if $u$ solves \eqref{extension of fracional laplacian} and it satisfies
		\begin{equation}\label{g bounded by V upper and lower}
			V(X + a_0 e_n) \leq u(X) \leq V(X + b_0 e_n) \quad \text{in } B_\rho(X^*) \subset B_1,
		\end{equation}
		with $V = V_{M, \xi', a, b} \in \Vv^0_\delta$ and
		\[
		\bar{C}\delta^2 \leq \frac{b_0 - a_0}{\rho} \leq \bar{\epsilon},
		\]
		with $|a_0|, |b_0| \leq 1$, then
		\begin{equation}\label{conclusion of thm3.8}
			V(X + a_1 e_n) \leq u(X) \leq V(X + b_1 e_n) \quad \text{in } B_{\bar{\eta} \rho}(X^*),
		\end{equation}
		with
		\[
		a_0 \leq a_1 \leq b_1 \leq b_0, \quad b_1 - a_1 = (1 - \bar{\eta})(b_0 - a_0),
		\]
		for a small universal constant $\bar{\eta} \in (0, 1/2)$.
	\end{theorem}
	\begin{proof}
		To begin with, we translate and rescale $u,V$ to simplify the result. First, we translate the origin such that the flatness hypothesis \eqref{g bounded by V upper and lower} in $B_\rho(X^*) \subset B_2$ with
		\[
		(X^*)' = 0, \quad a_0 + b_0 = 0, \quad V \in \Vv_{2\delta}^0.
		\]
		We dilate the picture by a factor of $2/\rho$. Then it suffices to consider the case $\rho=2$. We set
		\[
		u_\rho(X) = \left( \frac{\rho}{2} \right)^{-s} u\left( \frac{\rho}{2} X \right), \quad V_\rho(X) = \left( \frac{\rho}{2} \right)^{-s} V\left( \frac{\rho}{2} X \right),
		\]
		which are defined in a ball of radius 2 included in $B_{4/\rho}$. Notice that, if $V \in \Vv_{2\delta}^0$ then $V_\rho \in \Vv_{2\delta}^0$.
		
		For simplicity of notation, we drop $\rho$ and use $g_\rho,V_\rho$ to denote $g,V$. After a translation, we have $\eqref{g bounded by V upper and lower}\text{ in some ball }B_2(X^*) \subset \mathbb{R}^{n+1}.$ 
		\[V \in \Vv_0^{2\delta}, \;a_0 = -\epsilon, \;b_0 = \epsilon, \;(X^*)' = 0,\;\bar{C}\delta^2 \leq \epsilon \leq \bar{\epsilon}.\] 
		It suffices to prove the result \eqref{conclusion of thm3.8} in a ball $B_{2\bar{\eta}}(X^*)$.
		
		To prove the desired result, we consider the interior case and boundary case seperately. We distinguish three cases depending on whether $X^*$ is close to $\Ll$, close to $\pP$ but away from $\Ll$, or far from $\pP$. Here $\pP$ and $\Ll$ are defined in \eqref{half plane P}.
		
		By Proposition \ref{compare Vn Un}, we have the equation\begin{equation}\label{condition in harnack}
			c \leq \frac{V_n}{U_n} \leq C, \quad |(\Delta+\frac{\beta}{y}\partial_y) V| \leq C\delta^2 U_n, \quad \text{in } B_2(X^*) \setminus \left( \pP \cup \{|(x_n, y)| \leq 20\delta |x'|\} \right).
		\end{equation}In Case 2 and Case 3 we will need it. Furthermore, $\eta$ below is the universal constant from Proposition \ref{prop for Harnack}.
		
		\vspace{5pt}
		
		\textbf{Case 1.} $|X^*| < \eta/4$.
		
		\vspace{5pt}

		In this case, $B_1 \subset B_2(X^*)$. The assumptions of Proposition \ref{prop for Harnack} are satisfied. We conclude that for any $\bar{\eta} \leq \eta/4$ such that $B_{2\bar{\eta}}(X^*) \subset B_\eta$, either
		\[
		u(X) \leq V(X + (1 - \eta)\epsilon e_n),
		\]
		or
		\[
		u(X) \geq V(X - (1 - \eta)\epsilon e_n).
		\]
		Our conclusion holds for all $\bar{\eta} \leq \eta/4$.

		\vspace{5pt}
		
		\textbf{Case 2.} $|X^*| \geq \eta/4$, and $B_{\frac{\eta}{32}}(X^*) \cap \pP = \emptyset$.

		\vspace{5pt}
		
		As in the proof of Lemma \ref{lemma harnack}, by \eqref{condition in harnack} we have 
		\begin{equation}\label{function h}
			h(X) := u(X) - V(X - \epsilon e_n) \geq 0,
		\end{equation}
		\begin{equation}\label{equation of h}
			|(\Delta+\frac{\beta}{y}\partial_y) h| \leq C \delta^2 U_n, \quad \text{in } B := B_{\frac{\eta}{64}}(X^*),
		\end{equation}if $\bar{\epsilon}$ is small enough.
		
		Notice also that by Theorem \ref{interior harnack} Harnack inequality
		\[
		\frac{U_n(X)}{U_n(Y)} \leq C, \quad \text{for } X, Y \in B,
		\]
		with $C$ universal. Assume that
		\[
		u(X^*) \geq V(X^*).
		\]
		Then, in view of \eqref{condition in harnack},
		\[
		h(X^*) = u(X^*) - V(X^* - \epsilon e_n) \geq c \epsilon U_n(X^*).
		\]
		Hence, by Theorem \ref{interior harnack} Harnack inequality and the condition $C\bar{\delta}^2 \leq \epsilon$, we have
		\[
		h \geq c \epsilon U_n(X^*) - C \delta^2 \|U_n\|_{L^\infty(B)} \geq c' \epsilon U_n(X^*) \quad \text{in } B_{\frac{\eta}{128}}(X^*).
		\]
		Thus, using \eqref{condition in harnack} we have that for $\tau$ small enough
		\[
		h \geq c' \epsilon \sup_B V_n \geq V(X - (1 - \tau)\epsilon e_n) - V(X - \epsilon e_n) \quad \text{in } B_{\frac{\eta}{128}}(X^*),
		\]
		from which our desired conclusion follows with any $\bar{\eta}$ such that $2\bar{\eta} \leq \min\left\{\frac{\eta}{128}, \tau\right\}$.

		\vspace{5pt}
		
		\textbf{Case 3.} $|X^*| \geq \eta/4$ and $B_{\frac{\eta}{32}}(X^*) \cap \pP \neq \emptyset$.

		\vspace{5pt}
		
		The proof is similar to the proof in step 3 of Lemma \ref{lemma harnack}. We will construct two functions $q_1,q_2$ to estimate $h$. We will apply Theorem \ref{boundary harnack} boundary Harnack inequality near $\{y=0\}$.
		
		Assume that $X^* \in \{y > 0\}$ and call $X^*_0 = (x^*, 0)$ the projection of $X^*$ onto $\{y = 0\}$. Similar to \eqref{equation of h}, for the function $h$ in \eqref{function h} we have 
		\[
		|(\Delta+\frac{\beta}{y}\partial_y) h| \leq C \delta^2 U_n \quad \text{in } B := B_{\eta/8}(X_0^*) \cap \{y > 0\},
		\]
		for a universal constant $C$ if $\bar{\epsilon}$ is small enough.
		
		Denote by $Y^* = X_0^* + \frac{\eta}{16} e_n$. We prove the case $u(Y^*) \geq V(Y^*).$ For $u(Y^*)<V(Y^*)$, the proof is similar.
		As in the previous case, by Theorem \ref{interior harnack} Harnack inequality we have
		\begin{equation}
			\quad h \geq c \epsilon U_n(Y^*) \quad \text{in } B_{\eta/32}(Y^*).
		\end{equation}
		Now we argue similarly as in Lemma \ref{lemma harnack}.
		
		Denote by
		\[
		D := (B_{\eta/8}(X_0^*) \setminus B_{\eta/32}(Y^*)) \cap \{y > 0\}.
		\]
		To estimate $h$, we construct auxiliary functions $q_1,q_2$ satisfying in $D$
		\[
		(\Delta+\frac{\beta}{y}\partial_y) q_1 = 0, \quad (\Delta+\frac{\beta}{y}\partial_y) q_2 = -1,
		\]
		with boundary conditions respectively,
		\[
		q_1 = 1 \quad \text{on } \partial B_{\eta/32}(Y^*), \quad q_1 = 0 \quad \text{on } \partial(B_{\eta/8}(X_0^*) \cap \{y > 0\}),
		\]
		and
		\[
		q_2 = 0 \quad \text{on } \partial D.
		\]
		By the maximum principle, we obtain that
		\[
		h \geq c \epsilon U_n(Y^*) q_1 - C \delta^2 q_2 \quad \text{in } D.
		\]
		Moreover,
		\[
		q_1 \geq c q_2 \quad \text{in } D \cap B_{\eta/16}(X_0^*).
		\]
		Hence, using that $C \bar{\delta}^2 \leq \epsilon$, we get
		\[
		h(X) \geq c' \epsilon U_n(Y^*) q_1(X) \geq c \epsilon U_n(X) \quad \text{in } B_{\eta/16}(X_0^*) \cap \{y > 0\},
		\]
		where in the last inequality we used that $U_n(Y^*) q_1$ is comparable to $U_n$ in view of Theorem \ref{boundary harnack} boundary Harnack inequality.
		
		Finally, by \eqref{condition in harnack} and \eqref{comparable U_t} we conclude
		\[\begin{aligned}
			h(X) &= h(x, x_{n+1}) \geq c \epsilon \sup_{B_{\eta/8}(X_0^*)} U_n(y, x_{n+1}) \geq c \epsilon \sup_{B_{\eta/8}(X_0^*)} V_n(y, x_{n+1}),\\
			&\geq V(X - (1 - \tau)\epsilon e_n) - V(X - \epsilon e_n) \quad \text{in } B_{\eta/16}(X_0^*) \supset B_{\eta/32}(X^*).
		\end{aligned}\]
		Then our desired statement holds for $\bar{\eta} \leq \min\left\{\frac{\tau}{2}, \frac{\eta}{64}\right\}$.
	\end{proof}
	
	Finally, we prove the following corollary. It will be used to indicate the compactness of $u$ during the proof of $C^{2,\alpha}$ regularity.
	
	\begin{corollary}\label{continuity of limit}
		
		Let $u$ solve \eqref{extension of fracional laplacian} and satisfy for $\epsilon \leq \bar{\epsilon}$
		\[
		V(X - \epsilon e_n) \leq u(X) \leq V(X + \epsilon e_n) \quad \text{in } B_1,
		\]
		with
		\[
		V = V_{M, \xi', a, b} \in \Vv_0^\delta, \quad \bar{C}\delta^2 \leq \epsilon,
		\]
		for $\bar{\epsilon}, C > 0$ universal constants. If $\epsilon \leq \bar{\epsilon} \bar{\eta}^{2m_0}$
		for some nonnegative integer $m_0$ (with $\eta > 0$ small universal), then the function
		\[\tilde{u}_{\epsilon,V}(X):=\frac{\tilde{u}(X)-\tilde{V}(X)}{\epsilon},\quad X\in B_{1-4\delta-\epsilon}\setminus \pP,\] satisfies
		\[
		a_\epsilon(X) \leq \tilde{u}_{\epsilon,V}(X) \leq b_\epsilon(X) \quad \text{in } B_{\frac{1}{4}\bar{\eta}^{m_0}} \setminus \pP,
		\]
		with
		\[
		b_\epsilon - a_\epsilon \leq 2(1 - \bar{\eta})^{m_0},
		\]
		and $a_\epsilon, b_\epsilon$ having a modulus of continuity bounded by the H\"{o}lder function $o t^\gamma$, for $o, \gamma$ depending only on $\bar{\eta}$.
	\end{corollary}
	\begin{proof}
		To begin with, we translate $u$ such that \( u \) satisfies \eqref{g bounded by V upper and lower} in \( B_1 \) with \( a_0 = -\epsilon \) and \( b_0 = \epsilon \) for some small \( \epsilon \ll \bar{\epsilon} \), and \( \delta \) such that \(\bar{C}\delta^2 \leq \epsilon \). By Remark \ref{delta flatness remark}, the functions \( V \) and \( g \) are \( (4\delta + \epsilon) \)-flat in \( B_1 \).
		
		Then, at any point \( X^* \in B_{1/2} \), we can apply Theorem \ref{g V harnack} repeatedly for a sequence of radii \( \rho_m = \frac{1}{2} \bar{\eta}^m \) and obtain
		\[
		V(X + a_m e_n) \leq u(X) \leq V(X + b_m e_n) \quad \text{in } B_{\frac{1}{2} \bar{\eta}^m}(X^*),
		\]
		with
		\begin{equation}\label{distance between am and bm}
			b_m - a_m = (b_0 - a_0)(1 - \bar{\eta})^m = 2\epsilon(1 - \bar{\eta})^m,
		\end{equation}
		for all \( m \geq 1 \) such that
		\[
		4\epsilon\frac{(1 - \bar{\eta})^{m-1}}{\bar{\eta}^{m-1}}   \leq \bar{\epsilon}.
		\]
		This implies that for all such \( m \), the function \( \tilde{g} \) satisfies
		\[
		\tilde{V} + a_m \leq \tilde{u} \leq \tilde{V} + b_m \quad \text{in } B_{\frac{1}{2} \bar{\eta}^m - 4\delta - \epsilon}(X^*) \setminus \pP,
		\]
		with \( a_m \) and \( b_m \) as in \eqref{distance between am and bm}.
		
		We then get that in \( B_{\frac{1}{4} \bar{\eta}^m}(X^*) \setminus \pP \),
		\[
		\text{osc } \tilde{u}_{\epsilon, V} \leq 2(1 - \bar{\eta})^m,
		\]
		provided that
		\[
		4\delta + \epsilon \leq \epsilon^{1/2} \leq \bar{\eta}^{m/4}.
		\]
		If \( \epsilon \leq \bar{\epsilon} \frac{\bar{\eta}}{2m_0} \) for some nonnegative integer \( m_0 \), then our inequalities above  hold for all \( m \leq m_0 \). We thus obtain the following corollary.
		
	\end{proof}
	
	\section{$C^{2,\alpha}$ Regularity}\label{C2alpha regularity}
	In this section, we will prove the $C^{2,\alpha}$ regularity by using a sequence of paraboloids to approximate $F(u)$ following the spirit of \cite{DeSilva2012b} by D. De Silva and O. Savin.	To begin with, we consider a sequence of paraboloids approximating the free boundary and analyze the limit. For $\beta=0$, it is easy to analyze since the operator $L_\beta=\Delta$. For general $\beta\in (-1,1)$, we will need a more precise analysis.
	
	\begin{lemma}\label{limit of w0}
		Suppose there exists a sequence $\lim_{k\to 0}\lambda_k=0$ such that \[
		V_{\lambda_k} = V_{\lambda M,0,\lambda a, \lambda b} \in \Vv^0_{\lambda_k},
		\]\
		and\
		\begin{equation}\label{scaled condition}
			V_{\lambda_k}(X - \lambda_k^{1 + \alpha} e_n) \leq u_{\lambda_k}(X) \leq V_{\lambda_k}(X + \lambda_k^{1 + \alpha} e_n) \quad \text{in} \ B_1.
		\end{equation}
		For $w_{\lambda_k} := \frac{\tilde{u}_{\lambda_k} - \gamma_{V_{\lambda_k}}}{\lambda_k^{1+\alpha}},\;\tilde{u}_{\lambda_k}$ defined in Definition \ref{domain variation definition} and $\gamma_{V_{\lambda_k}}$ is defined in \eqref{gamma definition}, we have that $w_0=\lim_{k\to \infty} w_{\lambda_k}$ exists and satisfies \eqref{linearized problem}.
	\end{lemma}
	
	\begin{proof}
		We drop $k$ to simplify the notations. From Proposition \ref{difference between tildeV and gammaV}, we have
		\begin{equation}\label{uniformly convergence}
			w_{\lambda} = \frac{\tilde{u}_{\lambda} - \tilde{V}_{\lambda}}{\epsilon} + \frac{\tilde{V}_{\lambda} - \gamma_{V_{\lambda}}}{\epsilon} = (\tilde{u}_{\lambda})_{\epsilon, {V}_{\lambda}} + O\left(\epsilon^{1-\alpha}\right).
		\end{equation}
		The second term tends to $0$ as $\lambda$ tends to $0$. By Corollary \ref{continuity of limit}, the first term is uniformly H\"{o}lder continuous. Furthermore. $w_0(0) = 0$ and $|w_0| \leq 1$.
		
		We show that $w_0$ solves \eqref{linearized problem} according to Definition \ref{definition of solution to linearized problem}. We prove in 2 steps. In step 1,  we prove that $L_\beta(U_nw_0)=0$ in $B_{1/2}\setminus \pP$ and show $w_0\in C^{1,1}_{loc}(B_{1/2}\setminus\pP)$. In step 2, we prove the boundary condition $|\nabla_r w_0|=0$ on $\Ll$. Here $\pP$ and $\Ll$ are defined in \eqref{half plane P} and $r:=\text{dist}(X,\Ll)$.
		
		\vspace{5pt}
		
		\textit{Step 1:  $L_\beta(U_nw_0)=0$ in $B_{1/2}\setminus \pP$ and $w_0\in C^{1,1}_{loc}(B_{1/2}\setminus\pP)$.}
		
		\vspace{5pt}
		
		Suppose the convergence is $C^{1,t}$ for $t$ close to $1$, it is easy to verify that $L_\beta(U_nw_0)=0$ in $B_{1/2}\setminus \pP$. In this step, we verify the regularity condition $w\in C^{1,1}_{loc}(B_{1/2}\setminus \pP)$. We use $u_\lambda-V_\lambda$ to approximate $U_t(\tilde{u}_\lambda-\tilde{V}_\lambda)$ for the desired result. We first prove it is $C^{1,t}$ for a given $t\in(0,1)$, then raise the regularity to $C^{1,1}$ by the property of the equation. 
		
		By \eqref{uniformly convergence}, we have that $(\tilde{u}_{\lambda})_{\epsilon, {V}_{\lambda}}$ converges to $w_0$ uniformly. As $L_\beta u_\lambda =0$, by \eqref{estimate of Lbeta} we have
		\begin{equation}\label{equation of difference}
			|L_\beta(u_\lambda-V_\lambda)|\leq|y|^\beta C_0\lambda^2U_t.
		\end{equation}  For $d=$dist$(X,\pP)$, Theorem \ref{schauder estimate} implies \[\|u_\lambda-V_\lambda\|_{C^{1,t}(B_{d/4}(X))}\leq C(\|u_\lambda-V_\lambda\|_{L^\infty(B_{d/2}(X))}+\lambda^2\|U_t\|_{L^\infty(B_{d/2}(X))})\quad\text{for given }t\in (0,1).\]
		We will estimate the right hand side to derive the desired regularity. For the first term, we have \[|u_\lambda-V_\lambda|(X)\leq |V_\lambda(X+\lambda^{1+\alpha}e_n)-V_\lambda(X)|\leq C\lambda^{1+\alpha},\] for $\lambda$ small. For $\lambda$ tending to $0$, The second inequality is based on that $V_n$ converges to $U_t(x_n,y)$ uniformly in $B_{d/2}(X)$.
		
		To estimate the second term, we have $U_t(t,y)$ converges to $U_n(x_n,y)$ uniformly in $B_{d/2}(X)$ for $\lambda$ tending to $0$. Therefore $U_t$ is uniformly bounded in $B_{d/2}(X)$ for $\lambda$ small. Now we have\[\|u_\lambda-V_\lambda\|_{C^{1,t}(B_{d/4}(X))}\leq C\lambda^{1+\alpha}.\]By Arzela-Ascoli theorem, $\frac{u_\lambda-V_\lambda}{\epsilon}$ converges to a $C^{1,t}_{loc}(B_1\setminus \pP)$ function for $\lambda\to 0$. 
		
		Next we raise the regularity to $C^{1,1}$.  We use \eqref{equation of difference} to derive that $|L_\beta(\frac{u_\lambda-V_\lambda}{\epsilon})|\leq|y|^\beta C_0\lambda^{1-\alpha}U_t$. Take $\lambda\to 0$, we have that $\frac{u_\lambda-V_\lambda}{\epsilon}$ converges to a solution to $L_\beta u=0$. Invoking Theorem \ref{regularity of harmonic}, the limit is in $C^{1,1}_{loc}(B_1\setminus \pP)$.
		
		Finally, we show that $\frac{u_\lambda-V_\lambda}{\epsilon }=\frac{{(u_\lambda(X)-U)}-{(V_\lambda(X)-U)}}{\epsilon U_n}U_n$ converges to $w_0U_n$ uniformly in $B_{d/2}(X)$. Since $U_n$ is smooth away from the line $\Ll$, we have $w_0\in C^{1,1}_{loc}(B_{d/2}(X))$.
		
		By definition and Taylor expansion, we have 
		\[u_\lambda(X)-U(X)=u_\lambda(X)-u(X-\tilde{u}_\lambda(X))=u_n(X)\tilde{u}_\lambda(X)-\frac{1}{2}u_{nn}(X-t\tilde{u}_\lambda(X))\tilde{u}_\lambda^2(X).\]
		By Proposition \ref{difference between tildeV and gammaV} and Theorem \ref{regularity of harmonic}, we know that\[|\tilde{u}|\leq C\lambda,\;|(u_\lambda)_n-U_n|\leq C\lambda,\; |u_{nn}|\leq C,\quad\text{in }B_1\]in $B_1$ for $\lambda$ small. It gives $u_\lambda(X)-U=U_n(X)\tilde{u}_\lambda(X)+O(\lambda^2)$. Similarly, we have $V_\lambda(X)-U(X)=U_n(X)\tilde{V}_\lambda(X)+O(\lambda^2)$. Combine these two inequalities together, we have $\frac{u_\lambda-V_\lambda}{\epsilon }=\frac{{(u_\lambda(X)-U)}-{(V_\lambda(X)-U)}}{\epsilon}U_n=(\tilde{u}_{\lambda})_{\epsilon, {V}_{\lambda}}U_n+C\lambda^{1-\alpha}$. Letting $\lambda$ go to $0$ and we have the desired result.
		
		\vspace{5pt}
		
		\textit{Step 2: $|\nabla_r w_0|= 0$ on $B_{1/2} \cap \Ll$.}
		
		\vspace{5pt}
		
		In this substep we prove that
		\[
		|\nabla_r w_0|(X_0) = 0, \quad X_0 = (x'_0,0,0) \in B_{1/2} \cap \Ll,
		\]
		in the viscosity sense of Definition \ref{definition of solution to linearized problem}.
		
		We argue by contradiction. Assume for simplicity (after a translation) that there exists a function $\varphi$ which touches $w_0$ from below at $0$ with $\varphi(0) = 0$ and such that
		\[
		\varphi(X) = \xi' \cdot x' + a r + O(|x'|^2 + r^{1+\alpha}),\quad a > 0.
		\]
		
		Then we can find constants $\sigma$, $\tilde{r}$ small and $A$ large, such that the polynomial
		\[
		q(X) = \xi' \cdot x' - \frac{A}{2} |x'|^2 + 2A(n-1)x_n r
		\]
		touches $\varphi$ from below at $0$ in a tubular neighborhood $N_{\tilde{r}} = \{|x'| \leq \tilde{r}, r \leq \tilde{r}\}$ of $0$, with
		\[
		\varphi - q \geq \sigma > 0, \quad \text{on } N_{\tilde{r}} \setminus N_{\tilde{r}/2}.
		\]
		Therefore we have
		\begin{equation}\label{difference away from origin}
			w_0 - q \geq \sigma > 0, \quad \text{on } N_{\tilde{r}} \setminus N_{\tilde{r}/2},
		\end{equation}
		and
		\[
		w_0(0) - q(0) = 0.
		\]
		
		In particular, by continuity near the origin, we can find a point $X^*$ such that
		\begin{equation}\label{difference near origin}
			w_0(X^*) - q(X^*) \leq \frac{\sigma}{8}, \quad X^* \in N_{\tilde{r}} \setminus \pP \text{ close to } 0.
		\end{equation}
		
		Define
		\[
		W_{\lambda} := V_{\lambda M + A\epsilon I, -\epsilon \xi', \lambda a, \lambda b + 2\epsilon A(n-1)} \in \Vv_{2\delta}.
		\]
		Then, in view of Proposition \ref{difference between tildeV and gammaV}, we have
		\[
		\tilde{W}_{\lambda} = \epsilon q + \gamma_{V_{\lambda}} + O(\delta^2).
		\]
		Moreover, $W_{\lambda}$ is a subsolution to our problem since $\epsilon \gg \delta^2$.
		
		Thus, from the uniform convergence of $w_{\lambda}$ to $w_0$ and \eqref{difference away from origin}, we get that (for all $\lambda$ small)
		\begin{equation}\label{difference near origin: contradiction}
			\frac{\tilde{g}_{\lambda} - \tilde{W}_{\lambda}}{\epsilon} = w_{\lambda} - q + O\left( \frac{\delta^2}{\epsilon} \right) \geq \frac{\sigma}{2} \quad \text{in } (N_{\tilde{r}} \setminus N_{\tilde{r}/2}) \setminus \pP.
		\end{equation}
		
		Similarly, from the uniform convergence of $w_{\lambda}$ to $w_0$ and \eqref{difference near origin}, we get that for $k$ large
		\[
		\frac{\tilde{u}_{\lambda} - W_{\lambda}}{\epsilon}\leq \frac{\sigma}{4}\text{ in }\tilde{N}_r \setminus \pP.
		\]
		On the other hand, it follows from Lemma \ref{maximum principle} and \eqref{difference near origin} that
		\[
		\frac{\tilde{u}_{\lambda} - W_{\lambda}}{\epsilon}\geq \frac{\sigma}{2}\text{ in }\tilde{N}_r \setminus \pP,
		\]
		which contradicts \eqref{difference near origin: contradiction}.
		
	\end{proof}

	Now we prove the following quadratic improvement of flatness proposition. If a solution \( g \) stays in a \( \lambda^{2+\alpha} \) neighborhood of a function \( V \in \Vv_1^0 \) in a ball \( B_\lambda \), then in \( B_{\eta\lambda} \), \( g \) is in a \( (\lambda\eta)^{2+\alpha} \) neighborhood of another function \( V \) in the same class.
	
	\begin{proposition}\label{proposition C2alpha}
		Given \( \alpha \in (0, 1) \), there exist \( \lambda_0, \eta_0 \in (0, 1) \) and \( C > 0 \) large, depending on \( \alpha \), $\beta$, and \( n \). If \( u \) solves \eqref{extension of fracional laplacian}, \( 0 \in F(u) \), and \( u \) satisfies
		\begin{equation}\label{condition for C2 alpha prop interation}
			V(X - \lambda^{2+\alpha} e_n) \leq u(X) \leq V(X + \lambda^{2+\alpha} e_n), \quad \text{in } B_\lambda \text{ with } 0 < \lambda \leq \lambda_0,
		\end{equation}
		for \( V = V_{M,0,a,b} \in \Vv_1^0 \). Then in a possibly different system of coordinates denoted by
		\[
		\bar{E} = \{\overline{e}_1, \ldots, \overline{e}_n, \overline{e}_{n+1} \},
		\]
		we have
		\begin{equation}\label{result for C2 alpha prop interation}
			\bar{V}(X - (\eta_0\lambda)^{2+\alpha} \overline{e}_n) \leq u(X) \leq \bar{V}(X + (\eta_0\lambda)^{2+\alpha} \overline{e}_n), \quad \text{in } B_{\eta_0 \lambda},
		\end{equation}
		for some \( \bar{V} = \bar{V}_{\bar{M},0, \bar{a}, \bar{b}} \)  with \( \bar{\Ss} \) given in the \( \bar{E} \).
		
		Moreover, for $\bar{V}$ we have\[
		\|\bar{M} - M\|, |\bar{a} - a|, |\bar{b} - b| \leq C \lambda^\alpha, \quad \frac{1}{2s}\bar{a} + \bar{b} - \text{tr}(\bar{M}) = 0.
		\]And the surfaces $\Ss,\bar{\Ss}$ given by
		\[\Ss = \left\{ x_n = \frac{1}{2} (x')^T M x \right\},\; \bar{\Ss} = \left\{ \overline{x}_n = \frac{1}{2} (\overline{x}')^T \bar{M} \overline{x}' \right\},\]
		 separate in \( B_\sigma \) at most \( C(\lambda^\alpha \sigma^{2} + \lambda^{1+\alpha} \sigma) \) for any \( \sigma \in (0, 1] \).
	\end{proposition}
	
	\begin{proof}
		We will prove the desired result in by contradiction. Assume that no such $\lambda_0$ exists, then, we can find
		a sequence $\lim_{k\to 0}\lambda_k=0$ such that $g_k$ and $V_k$ satisfying \eqref{condition for C2 alpha prop interation} and \eqref{result for C2 alpha prop interation} fails. 
		
		We show the contradiction in two steps. In step 1, we rescale the functions $u$ and $V$ and analyze the limit of function $w_{\lambda_k} := \frac{\tilde{u}_{\lambda_k} - \gamma_{V_{\lambda_k}}}{\lambda_k^{1+\alpha}},\;\tilde{u}_{\lambda_k}$ defined in Definition \ref{domain variation definition} and $\gamma_{V_{\lambda_k}}$ is defined in \eqref{gamma definition}. In step 2, we apply Theorem \ref{quadratic expansion} to construct $\bar{V}$ to derive the contradiction.
		
		\vspace{5pt}
		
		\textit{Step 1: Rescaling and $w_\lambda$.}
		
		\vspace{5pt}
		
		To begin with, we rescale $g_k$ and $V_k$. For simplicity of notation, we drop the dependence on $k$ and
		denote
		\[
		u_{\lambda}(X) = \lambda^{s} u(\lambda X), \quad V_{\lambda}(X) = \lambda^{s} V(\lambda X), \quad X \in B_1.
		\]
		Now we have
		
		\begin{equation}\label{scaled condition}
			V_{\lambda_k}(X - \lambda_k^{1 + \alpha} e_n) \leq u_{\lambda_k}(X) \leq V_{\lambda_k}(X + \lambda_k^{1 + \alpha} e_n) \quad \text{in} \ B_1.
		\end{equation}
		From Lemma \ref{limit of w0}, we have that $w_0=\lim_{k\to \infty} w_{\lambda_k}$ is a solution to \eqref{linearized problem}

		\vspace{5pt}
		
		\textit{Step 2: Construction of $V$.}
		
		\vspace{5pt}
		
		In this step, we construct a new function $\bar{V}$ such that $\tilde{g}_\lambda$ can be approximated by $\bar{V}$ for $\lambda$ small. We will use the uniform convergence of $w_\lambda$ to $w_0$ to construct a polynomial approximating it. By Proposition \ref{difference between tildeV and gammaV} we can construct a solution related to the polynomial. After a modification of the coordinate system, we have constructed $\bar{V}$.
		
		In step 2, we have proven that $w_0$ is a solution to \eqref{linearized problem}. By Theorem \ref{quadratic expansion}, we have
		\[|w_0(X)-(\xi_0\cdot x'+\frac{1}{2}(x')^TM_0x'-\frac{1}{2s}(\frac{a_0}{2}r^2+b_0x_nr))|\leq C(|x-x'|^3+r^3)\quad\text{in }B_{1/2}.\]
		Then for $\eta_0$ small enough depending on $\alpha$, we have
		\[|w_0(X)-(\xi_0\cdot x'+\frac{1}{2}(x')^TM_0x'-\frac{1}{2s}(\frac{a_0}{2}r^2+b_0x_nr))|\leq \frac{1}{4}\eta_0^{2+\alpha}\quad\text{in }B_{4\eta_0}.\]
		
		Set $T_\lambda:=V_{\lambda M-\epsilon M_0,\epsilon\xi_0,\lambda a-\epsilon a_0,\lambda b-\epsilon b_0}$. Here we dropped $k$ to simplify the notation. By Proposition \ref{difference between tildeV and gammaV}, we have
		\[\begin{aligned}
			\tilde{T}_\lambda-\gamma_{V_\lambda}&=\tilde{T}_\lambda-\gamma_{T_\lambda}+\gamma_{T_\lambda}-\gamma_{V_\lambda}=O(\lambda^2)+\gamma_{T_\lambda}-\gamma_{V_\lambda}\\
			&=O(\lambda^2)-\epsilon\frac{1}{2s}(\frac{a_0}{2}r^2+b_0rx_n)+\epsilon\xi_0\cdot x'+\frac{1}{2}\epsilon(x')^TM_0x'.
		\end{aligned}\]
		
		From the uniform convergence of $w_\lambda$ to $w_0$, t follows that for fixed $r$ and sufficiently large $k$, $\lambda_k$ satisfies
		\[|w_\lambda-\frac{\tilde{T}_\lambda-\gamma_{V_\lambda}}{\epsilon}|=|\frac{\tilde{u}_\lambda-\tilde{V}_\lambda}{\epsilon}|\leq \frac{1}{2}\eta_0^{2+\alpha}\quad\text{in }B_{4\eta_0}.\]
		
		By the definition of $w_\lambda$, we have
		\[\tilde{T}_\lambda-\frac{1}{2}\epsilon\eta_0^{2+\alpha}\leq \tilde{u}_\lambda\leq \tilde{T}_\lambda+\frac{1}{2}\epsilon\eta_0^{2+\alpha}\quad\text{ in }B_{2\eta_0}.\]
		
		Rescale $u_\lambda$  back from $B_1$ to $B_\lambda$ and obtain
		
		\begin{equation}\label{g between T}
			T(X-\frac{1}{2}\epsilon\lambda\eta_0^{2+\alpha}e_n)\leq u(X)\leq T(X+\frac{1}{2}\epsilon\lambda\eta_0^{2+\alpha}e_n)\quad\text{ in }B_{2\lambda\eta_0}.
		\end{equation}
		
		Finally, we show that $\Ss$ and $\bar{\Ss}$ separate. Then \eqref{result for C2 alpha prop interation} follows by definition. We will consider the function $T=V_{M_T,0,a_T,b_T}$ and the surface  \(\Ss_T:=\{x_n:=\frac{1}{2}(x')^TM_Tx'+\xi_T\cdot x'\},\) with\[ M_T=M-\frac{\epsilon}{\lambda}M_0,\;a_T=a-\frac{\epsilon}{\lambda}a_0,\;b_T=b-\frac{\epsilon}{\lambda}b_0.\]The distance between $\Ss$ and $\Ss_T$ can be derived by simple calcuius. It suffices to show that $\Ss_T$ approximates  $\bar{\Ss}$.

		To begin with, we construct a different system of coordinates $\bar{E}$. Without loss of generalization, we assume $\xi_0$ is in $e_1$ direction. We choose $\bar{E}={\bar{e}_1,\cdots,\bar{e}_{n+1}}$ with $\bar{e}_i=e_i$ for $i\neq 1,n$, and $\bar{e}_n$ normal  to $\epsilon\xi_0\cdot x'$ at $0$. $\bar{E}$ is obtained by an orthogonal tranformation $O$ such that $\bar{x}=Ox$. Furthermore, we have that $\|O-I\|\leq C\epsilon$.
		
		By definition, we have that \[\bar{\Ss}=\{\bar{x}_n=\frac{1}{2}(\bar{x}')^TM_T\bar{x}'\}.\]In the original system $E$, there exists a function $g$ such that\[\bar{\Ss}=\{x_n=g(x')\}.\]
		By the property of $O$, we have that\[\|D_{x'}^2g-M_T\|_{L^\infty(B_1)}\leq C\epsilon,\quad\nabla_{x'}g(0)=\xi_T.\]
		Then in $B_{2\eta_0\lambda}$ we have that the surfaces $\Ss_T\pm \frac{\epsilon}{2}\lambda\eta_0^{2+\alpha}e_n$ lie between  $\bar{\Ss}\pm \epsilon\lambda\eta_0^{2+\alpha}\bar{e}_n$. We are done.
		
		Now we prove \eqref{result for C2 alpha prop interation}. We set $\bar{V}=V_{M_T,0,a_T,b_T}$ in $\bar{E}$. By the analysis above, we have the distance between $\bar{\Ss}$ and $\Ss_T$. Together with \eqref{g between T}, we have the desired result.
	\end{proof}

	Finally, we prove that flat free boundaries are $C^{2,\alpha}$ to end this section.
	\begin{theorem}\label{freeboundary C2alpha}
		For $\alpha\in (0,1)$, there exists $\bar{\epsilon}$ small depending on universal constants and $\alpha$, such that if $u$ solves \eqref{extension of fracional laplacian} in $B_1$ with\begin{equation}
			\{x \in B_1 : x_n \leq -\bar{\varepsilon} \} \subset \{x \in B_1 : u(x, 0) = 0\} \subset \{x \in B_1 : x_n \leq \bar{\varepsilon} \},
		\end{equation}
	then $F(u)\cap B_{1/2}^n$ is $C^{2,\alpha}$.
	\end{theorem}
	\begin{proof}
		Our goal is to construct a paraboloid that a paraboloid approximates  $F(u)$ with $C^{2,\alpha}$ precision at the origin. The argument can then be applied to any point in $F(u)\cap B_{1/2}$. The proof proceeds as follows. First we apply Proposition \ref{proposition C2alpha} iteratively to generate a sequence of approximating functions $V_k$ and coordinate systems  $E_k$. Second, we show that the corresponding surfaces $\Ss_k$
		converge uniformly to a limit paraboloid $\Ss$, which provides the desired approximation for $F(u)$. 
		
		To begin with, we consider the rescaling $u_\mu(X)=\mu^{-s}u(\mu X)$. By Lemma \ref{compare freeboundary with plane}, we have\[U(X-\tau^{2+\alpha}e_n)\leq u_\mu(X)\leq U(X+\tau^{2+\alpha}e_n),\quad\text{in }B_1,\]
		provided that $\bar{\epsilon},\mu$ small depending on $\tau\leq \lambda_0$. Here $\lambda_0$ is from Proposition \ref{proposition C2alpha} and $\tau$ is to be determined later.
		
		Applying Proposition \ref{proposition C2alpha} repeatedly to $u_\mu$, we have a sequence $\tau_k=\tau\eta_0^k$, $V_{\Ss_k,a_k,b_k}$, and vector $e_n^k$ such that\[V_{\Ss_k,a_k,b_k}(X-\tau_k^{2+\alpha}e_n^k)\leq u_\mu(X)\leq V_{\Ss_k,a_k,b_k}(X+\tau_k^{2+\alpha}e_n^k),\quad\text{in }B_{\tau_k}.\]
		We choose $\tau$ small such that $\sum C\tau_k^\alpha\leq 1$, therefore $M_k\to M_*$, $a_k\to a_*$, $b_k\to b_*$ always have norm no more than $1$. Furthermore, we have
		\[|a_k-a_*|,|b_k-b_*|\leq C\tau_k^\alpha,\] Then $F(u_\mu)$ is $C^{2,\alpha}$ at $0$. 
		
		From Proposition \ref{proposition C2alpha}, $\Ss_k$ and $\Ss_{k+1}$ separate in $B_\sigma$ at most $C(\tau_k^\alpha\sigma^2+\tau_k^{1+\alpha}\sigma)$, $\Ss_k$ converge uniformly to a paraboloid $\Ss_*$ in $B_1$. Since $\Ss_k$ and $\Ss_*$ separate in $B_{2\tau_k}$ at most $C\tau_k^{2+\alpha}$, we apply Lemma \ref{paraloid separate} to derive that
		\[V_{\Ss_*,a_*,b_*}(X-\tau_k^{2+\alpha}e_n^k)\leq u_\mu(X)\leq V_{\Ss_*,a_*,b_*}(X+\tau_k^{2+\alpha}e_n^k),\quad\text{in }B_{\tau_k}.\] Then $F(u_\mu)$ is $C^{2,\alpha}$ at $0$.
	\end{proof}
	\section{$C^\infty$ Regularity}\label{cinfty}
	In this section, we provide the proof of Theorem \ref{C infty of free boundary}. From Theorem \ref{C2alpha regularity}, we have that the free boundary is $C^{2,\alpha}$. A recent work \cite{Barrios2025} by B. Barrios, X. Ros-Oton, and M. Weidner provided a different proof for this step by considering a general type of global elliptic operator. We will prove using a different method. 
	
	In this section, we will show that the derivatives $g_i$ has the same regularity as $g$ as in \cite{Silva2015} by D. De Silva and O. Savin. In the following theorem, we consider the behavior of $\nabla_x^2u$ near the free boundary. Without loss of generality, we consider the following case throughout this section.
	\begin{equation}\label{Fu smooth condistion}
		F(u):=\{(x',g(x'),0)\},\quad\|g\|_{C^{k+2,\alpha}(B_{1/2})}\leq \frac{1}{2},\quad\nabla g(0)=0.
	\end{equation}
	\begin{theorem}\label{expansion theorem}
		If $u$ is a solution to \eqref{extension of fracional laplacian} with $F(u)$ satisfying \eqref{Fu smooth condistion} for $k\geq 0$, then we have:
		\begin{equation}\label{order0 expansion}
			\|\frac{u}{U_0}\|_{C^{k+1,\alpha}_{xr}(F(u)\cap B_{1/2})}\leq C, 
		\end{equation}
		\begin{equation}\label{order1 expansion}
			\|\frac{\nabla_xu}{U_0/r}\|_{C^{k+1,\alpha}_{xr}(F(u)\cap B_{1/2})}\leq C, 
		\end{equation}
		\begin{equation}\label{order2 expansion}
			\|\frac{\nabla_x^2u}{U_0/r^2}\|_{C^{k+1,\alpha}_{xr}(F(u)\cap B_{1/2})}\leq C,
		\end{equation}
		where $U_0$ is defined in \eqref{family of functions} and $C$ is a universal constant.
	\end{theorem}
	\begin{proof}
		Note that \eqref{order0 expansion} and \eqref{order1 expansion} follow from \cite[Proposition 3.2 and Lemma 7.2]{Jhaveri2017}. It suffices prove that \eqref{order2 expansion}. Without loss of generalization, we assume that the normal unit of $F(u)$ at $0$ is $e_n$.
		
		By \eqref{order0 expansion}, there exists a degree $k+1$ polynomial $P(x,r)$ such that $|u-U_0P|\leq CU_0r^{k+1+\alpha}$. We claim that for any $1\leq i,j\leq n$ we have
		\begin{equation}\label{point expansion}
			|\partial_{ij}(u-U_0P)|\leq C\frac{U_0}{r^2}|X|^{k+1+\alpha}\quad\text{in }\kK:=\{x_n^2+y^2\geq |x'|^2\}.
		\end{equation}
		Then for each $x\in F(u)$, we apply \eqref{point expansion} to \[\kK_x:=\{X\in\Rr^{n+1}:|(X-x)\cdot \nu|^2+y^2\geq |X-x|^2-|(X-x)\cdot \nu|^2\}.\] As $\bigcup_{x\in F(u)} \kK_x$ is the whole space, we have the desired result.

		We prove \eqref{point expansion} by scaling. To begin with, we define the following scalings for $\lambda\in(0,1)$:
		\[\begin{aligned}
			\tilde{U}_\lambda(X)&:=\frac{(u-U_0P)(\lambda X)}{\lambda^{k+1+\alpha+s}},\quad &\pP_\lambda:=\{x:u(\lambda x)=0\},\\
			r_\lambda(X)&:=\frac{\text{dist}(\lambda X,F(u))}{\lambda},\quad &U_{0,\lambda}:=\frac{U(\lambda X)}{\lambda^s}.
		\end{aligned}\]
		By choosing different $\lambda$, $\kK\cap(B_{3\lambda/4}\setminus \bar{B}_{\lambda/4})$ covers $\kK$. Therefore it suffices to show that 
		\begin{equation}\label{point expansion lambda}
			|\nabla_{x}^2\tilde{U}_\lambda|\leq CU_{0,\lambda}\quad\text{in }\kK\cap(B_{3/4}\setminus \bar{B}_{1/4}),\quad\forall\lambda\in(0,1).
		\end{equation}
		
		We will prove \eqref{point expansion lambda} in 3 steps. In step 1, we determine the equation satisfied by $\tilde{U}_\lambda$. In step 2 and 3, we prove the desired result for $x_n\geq 0$ and $x_n<0$ using Theorem \ref{interior harnack} Harnack inequality and Theorem \ref{boundary harnack} boundary Harnack inequality.
		
		\vspace{5pt}
		
		\textit{Step 1: equations for $\tilde{U}$.}
		
		By \eqref{order0 expansion}, $\|\tilde{U}_\lambda\|_{L^\infty(B_1)}\leq C$. First we consider $L_\beta\tilde{U}_\lambda$. Following \cite[Theorem 4.1]{Jhaveri2017}, we claim that
		\[L_\beta\tilde{U}_\lambda=|y|^\beta F\text{ in }B_1\setminus \pP_\lambda,\quad |\frac{F(X)}{U_{0,\lambda}/r_\lambda}|\leq C|X|^{k+\alpha}.\] Here $F(X)=-\frac{U_{0,\lambda}}{r_\lambda}\frac{h(\lambda X)}{\lambda^{k+s}}$, $h(x,y)=\sum_{l=0}^{k}r^lh_l(x)$, $\|h_l\|_{C^{k,\alpha}(B_1^n)}\leq C(n,\beta)$ such that $h_l$ has vanishing derivatives up to order $k-l$ at zero. We rewrite the equation of $u$ to be
		\begin{equation}\label{modified equation}
			\Delta_x \tilde{U}=F-(\frac{\beta}{y}\partial_y+\partial_{yy})\tilde{U}.
		\end{equation}
		
		We prove the claim to end this step. It suffices to prove for the case $\lambda=1$. The result for general $\lambda$ follows by scaling. As $L_\beta u=0$, we  only need to consider $L_\beta (U_0P)$. By calculation following \cite{Jhaveri2017}, we have
		\[\begin{aligned}
			L_\beta(U_0x^\mu r^m)&=|y|^\beta\frac{U_0}{r}(-(dm+sr)\kappa x^\mu r^{m-1}+m(m+1)x^\mu r^{m-1}\\
			&+2r^{m-1}(dm+sr)\nu\cdot\nabla_xx^\mu+\mu_i(\mu_i-1)x^{\mu-2\bar{i}}r^{m+1})\\
			&=|y|^\beta\frac{U_0}{r}(m(m+1+2\mu_n)x^\mu r^{m-1}+2s\mu_n x^{\mu-\bar{n}}r^m\\
			&+\mu_i(\mu_i-1)x^{\mu-2\bar{i}}r^{m+1}+c_{\sigma l}^{\mu m}x^\sigma r^l+h^{\mu m}(x,r)),
		\end{aligned}\]
		where $\kappa$ is defined in \eqref{Laplacian t}. Here $c_{\sigma l}^{\mu m}\neq 0$ only if $|\mu|+m\leq|\sigma|+l\leq k$, and
		\[h^{\mu m}(x,r)=r^mh^\mu_m(x)+mr^{m-1}h^\mu_{m-1}(x),\]
		$ h^\mu_m,h^\mu_{m-1}\in  C^{k,\alpha}(B_1^n)\text{ has vanishing derivatives of order $k-m$ and $k-m+1$ at zero.}$

		\vspace{5pt}
		
		\textit{Step 2: \eqref{point expansion lambda} in $\{x_n\geq 0\}$.}
		
		\vspace{5pt}
		
		Since $\|F(u)\|_{C^{k+2}}\leq 1 $, $U_{0,\lambda}$ has a positive lower bound in $\kK\cap(B_{3/4}\setminus \bar{B}_{1/4})\cap\{x_n\geq 0\}$. We only need to show that $\nabla_{x}^2\tilde{U}$ is uniformly bounded.
		
		We claim that the right hand side of \eqref{modified equation} is $C^\alpha$ bounded and apply Schauder estimate to \eqref{modified equation}. The desired result will follow.
		
		The regularity of $F$ is natural. We focus on the regularity of $\frac{\beta}{y}\partial_y\tilde{U}+\partial_{yy}\tilde{U}$. By calculus, we have \[L_{-\beta}(|y|^\beta\partial_y \tilde{U})=\partial_y F.\] The transformation was pointed out in \cite{Caffarelli2007}. We apply Prop. \ref{interior regularity for nabla_xu and uy} to the equation, it gives 
		\[\||y|^\beta\partial_y \tilde{U}\|_{L^\infty(\kK\cap(B_{7/8}\setminus \bar{B}_{1/8})\cap\{x_n\geq 0\})}\leq C.\]Similarly, we have that\[L_{\beta}(|y|^{-\beta}\partial_y(|y|^\beta\partial_y \tilde{U}))=\partial_y(|y|^\beta\partial_y F).\] We apply Proposition \ref{interior regularity for nabla_xu and uy} again to derive that 
		\[\||y|^{-\beta}\partial_y(|y|^\beta\partial_y \tilde{U})\|_{L^\infty(\kK\cap(B_{7/8}\setminus \bar{B}_{1/8})\cap\{x_n\geq 0\})}\leq C.\] We apply the same trick for the third time to derive
		\[\||y|^\beta\partial_y(|y|^{-\beta}\partial_y(|y|^\beta\partial_y \tilde{U}))\|_{L^\infty(\kK\cap(B_{7/8}\setminus \bar{B}_{1/8})\cap\{x_n\geq 0\})}\leq C.\]
		
		As $|y|^\beta\partial_y(|y|^{-\beta}\partial_y(|y|^\beta\partial_y \tilde{U}))=|y|^\beta\partial_y(\frac{\beta}{y}\partial_y\tilde{U}+\partial_{yy}\tilde{U})$, the inequality gives the  H\"{o}lder continuity of $\frac{\beta}{y}\partial_y\tilde{U}+\partial_{yy}\tilde{U}$ in the $y$ direction. For the H\"{o}lder continuity in $x$-direction, we consider the equation $$L_\beta\partial_x(|y|^{-\beta}\partial_y(|y|^\beta\partial_y u))=\partial_x\partial_y(|y|^{-\beta}\partial_y F).$$ We apply Prop. \ref{interior regularity for nabla_xu and uy} to the equation, it gives that\[\|\partial_x(\frac{\beta}{y}\partial_y\tilde{U}+\partial_{yy}\tilde{U})\|_{L^\infty(\kK\cap(B_{7/8}\setminus \bar{B}_{1/8})\cap\{x_n\geq 0\})}\leq C.\]
		
		\vspace{5pt}
		
		\textit{Step 3: \eqref{point expansion lambda} in $\{x_n< 0\}$.}
		
		\vspace{5pt}
		
		We notice that, if we can prove the inequlity \eqref{point expansion lambda} near $\{y=0\}$, then we only need to check boundness for points away from $\{y=0\}$. That was proven in Step 2. 
		
		As $\partial_{ij} \tilde{U}=0$ on $\{y=0\}$, it suffices to show  that $|y|^\beta\partial_y(\partial_{ij}\tilde{U})$ is bounded. We consider
		\[L_{-\beta}\partial_{ij}(|y|^\beta\partial_y\tilde{U})=\partial_{ij}\partial_y F.\]
		Apply  Proposition \ref{interior regularity for nabla_xu and uy} to the equation, we have
		\[\|\partial_{ij}(|y|^\beta\partial_y\tilde{U})\|_{L^\infty(\kK\cap(B_{7/8}\setminus \bar{B}_{1/8})\cap\{x_n\geq 0\})}\leq C.\]
	\end{proof}
	
	Now we try to improve the regularity of $F(u)$. We want to show that the derivatives of the free boundary has the same regularity as itself. To prove that, we need to analyze the PDE of it. We notice that $-\frac{u_i}{u_n}$, $i\leq n-1$ is a natural extension of the derivative. It satisfies $L_\beta(u_n\frac{u_i}{u_n})=0$. To analyze the regularity of $\frac{u_i}{u_n}$, we state the following theorem.
	\begin{theorem}\label{regularity improvement}
		Let $u$ be a solution to \eqref{extension of fracional laplacian} and $\|F(u)\|_{C^{k+2,\alpha}(B_1)}\leq \frac{1}{2}$. Assume $w\in C(B_1)$, even in $y$, satisfies the equation \begin{equation}\label{equation for gi}
			\left\{
			\begin{aligned}
				&L_\beta(u_n w) = 0 \quad &\text{in }& B_1 \setminus \{u>0\},\\
				&|\nabla_r w| = 0 \quad &\text{on }& F(u). 
			\end{aligned}\right.
		\end{equation}
		Then $\|w\|_{C^{k+2,\alpha}_{xr}(F(u)\cap B_{1/2})}\leq C$ with $C$  depending only on $n,k,\alpha.$
	\end{theorem}
	
	The case $F(u)=\Ll$ in \eqref{half plane P} is Theorem \ref{quadratic expansion}. For the general case, we will use iteration to construct a sequence of $C^{k+2}$ functions approximating $w$. To begin with, we follow \cite[Section 7.1]{Silva2015} to define $E(Q)$. For $z=(z',z_n)\in \Rr^{n-1}\times\Rr$, we set
	\[T_0(z,0)=(z',g(z'),0)+z_n\frac{(-\nabla g,1,0)}{\sqrt{1+|\nabla g|^2}}.\]
	Since $\nu$ has lower regularity, $T_0$ only has $C^{k+1,\alpha}$ regularity. However, the restriction of $T$ on the hyperline $\Ll$ in \ref{half plane P} is pointwise $C^{k+2,\alpha}$. Furthermore, $T$ is one-to-one near the origin. Then we define
	\begin{equation}\label{definition of Q}
		Q(x)=Q(z')=q_\mu z^\mu,\quad|\mu|\leq k+2,\quad q_\mu=0\;\text{if }\mu_n\neq 0.
	\end{equation}
	By the analysis above, $Q$ is pointwise $C^{k+2,\alpha}$ on $F(u)$. By Whitney extension, we have
	\begin{theorem}\label{whitney}(\cite[Thm. 7.1]{Silva2015} Whitney Extension Theorem)
		There exists $E(Q)$ such that
		\[D_x^\mu E(Q)=D_x^\mu Q\quad\text{on }F(u)\text{ for }\mu\leq k+2.\]
		and
		\[\|E(Q)\|_{C^{k+2,\alpha}_x(B_{1/2})}\leq C\|Q\|_{C^{k+2,\alpha}_x(F(u))}.\]
		Moreover, $E(Q)$ is linear in $Q$. For  $Q$ in \eqref{definition of Q}, we have
		\[E(Q)=\tilde{q}_\mu x^\mu+O(|x|^{k+2+\alpha}),\;\tilde{q}_\mu=q_\mu+\tilde{c}_\mu^\sigma q_\sigma,\;\tilde{c}_\mu^\sigma\neq 0\text{ if and only if }|\sigma|<|\mu|.\]
	\end{theorem}
	We can consider $E(Q)$ as $Q$ in Theorem \ref{quadratic expansion}. We want to have a  polynomial type function to approximate $u$. Furthermore, the property of $Q$ in Theorem \ref{quadratic expansion} implies that we need a function only depending on the projection onto $\Gamma$. However , that is impossible due to the regularity. Therefore we use Whitney extension to construct $E(Q)$. It extends $Q$ near $\Gamma$ in $C^{k+2.\alpha}$ regularity. Now we consider the equation for $E(Q)$.
	\begin{lemma}\label{equation for u_nEQ}
		
		Let \( e \) be a unit vector and \( u_e \) has the following expansion at \( 0 \),
		\[
		u_e = \frac{U_0}{r} \left(P^e_0 + O(|X|^{k+1+\alpha})\right),\quad\deg P^e_0\leq k+1.
		\]
		Then for polynomial $Q$ such that
		\[Q(y')=\sum_{\mu\in\Nn^{n-1}\times\{0\}^2}^{|\mu|\leq k+2}q_\mu y^\mu,\] we have
		\[
		L_\beta (u_e E(Q)) = |y|^\beta \frac{U_0}{r} (R + O(|X|^{k+\alpha}) \quad \text{in } B_1 \setminus \pP,
		\]where \( R \) is a polynomial of degree \( k \) in \( (x,r) \) given by
		\[
		R = A_{\sigma l} x^\sigma r^l,\quad |\sigma|+l\leq  k.
		\]
		Here $A_{\sigma l}$ satisfies
		\[A_{\sigma l}=\left\{\begin{aligned}
			&c^\mu_{\sigma l} q_\mu,\quad &(\sigma_n,l)&\neq (0,0),\\
			&c^\mu_{\sigma l}q_\mu+P^e_0(0) (\sigma_i + 1)(\sigma_i + 2) q_{\sigma + 2\bar{i}}   , \quad &(\sigma_n, l)&=(0,0).
		\end{aligned}\right.\]
		Additionally,
		\[
		c^\mu_{\sigma l} \neq 0 \quad \text{only if } |\mu| \leq |\sigma| + l + 1.
		\]
	\end{lemma}
	\begin{proof}
		Since $E(Q)$ is independent of $y$ direction and $L_\beta u_e=0$, $L_\beta(u_eE(Q))$ only have two main terms:
		\[L_\beta(u_eE(Q))=u_e|y|^\beta\Delta E(Q)+2|y|^\beta\nabla u_e\cdot\nabla E(Q).\]
		We will now find the expansion for each of these terms at the origin.
		
		First, we analyze the term involving $\Delta E(Q)$. By Theorem \ref{whitney}, we know it is pointwise $C^{k,\alpha}_x$ at the origin. We formally differentiate it at the origin to find the expansion. We have
		\[\Delta E(Q)(x)=\sum_{m=0}^{k}\frac{1}{m!}x_n^m\Delta\partial_n^mQ(0)+O(|x_n|^{k+\alpha})=\Delta Q(x)+O(|x_n|^{k+\alpha}).\]
		
		For the second term, we claim that
		\[
		\nabla E(Q) = \nabla Q + |d|^{k+s} \xi + |d|^{k+s+1} \eta
		\]
		for two bounded vectors \(\xi, \eta \in \mathbb{R}^n\) with \(\xi \cdot \nu = 0\).
		
		Assume for simplicity that \(x\) is a point on the \(e_n\)-axis. Then, since \(E(Q) \in C^{k+2,\alpha}_x\), we find
		\[
		\nabla E(Q)(x) = \sum_{m=0}^{k+1} \frac{1}{m!} x_n^m \nabla \partial_n^m Q(0) + O(|x_n|^{k+1+\alpha}).
		\]
		
		From Taylor expansion, we see that
		\[
		\sum_{m=0}^{k+1} \frac{1}{m!} x_n^m \nabla \partial_n^m Q(0) = \nabla Q(x) + \xi |d|^{k+s},
		\]
		for some bounded vector \(\xi\). Moreover, since \(Q\) is constant on perpendicular lines to \(\alpha\), we have \(\nabla Q(x) \cdot e_n = 0\) and \(\partial_n^l Q(0) = 0\) for all \(l \leq k+2\). Thus, the formula above gives \(\xi \cdot e_n = 0\), and our claim is proved.
		
		Next we give the estimate of  $u_{ei}$. By Theorem \ref{expansion theorem}, we have that\[\begin{aligned}
			u_{ei}&=\partial_i(\frac{U_0}{r}(P^e_0+O(|X|^{k+1+\alpha})))\\
			&=\nu^ns\frac{U_0}{r^2}(P^e_0+O(|X|^{k+1+\alpha}))-\nu^i\frac{U_0}{r^3}(P^e_0+O(|X|^{k+1+\alpha}))\\
			&+\frac{U_0}{r}(\partial_iP^e_0+O(|X|^{k+1+\alpha}))+\frac{U_0}{r^2}(d\nu^i\partial_r P^e_0+O(|X|^{k+1+\alpha}))\\
			&=\frac{U_0}{r^2}((s-\frac{d}{r})\nu^i P^e_0+r\partial_iP^e_0+d\nu^i\partial_rP^e_0+O(|X|^{k+1+s})).
		\end{aligned}\]
		
		Therefore we have that
		\[\nabla u_e\cdot\nabla E(Q)=\frac{U_0}{r}(\nabla_x P^e_0\cdot\nabla Q+O(|X|^{k+\alpha})),\]
		
		Finally, we derive the estimate of $L_\beta(u_eE(Q))$.
		\[\begin{aligned}
			L_\beta(u_eE(Q))=\frac{U_0}{r}(P^e_0\cdot(\mu_i+2)(\mu_i+1) q_{\mu+2\bar{i}}x^{\mu}+\nabla_x P^e_0\cdot\nabla Q+O(|X|^{k+\alpha})).
		\end{aligned}\]
		To obtain the desired result, we only need to apply the expansion of $P_0^e$ and $\nabla P^e_0$ at $0$ to the equality above.
	\end{proof}
	
	Next we prove the compactness of solutions to \eqref{linearized problem}. Suppose that \( u \) is a solution to \eqref{extension of fracional laplacian}, by Theorem \ref{expansion theorem}, we have
	\[u=U_0(1+O(|X|)),\quad u_n=\frac{U_0}{r}(s+O(|X|)).\]
	We consider the scaling
	\[u_\delta(x)=\delta^{-s}u(\delta x),\]
	it satisfies the expansion
	\begin{equation}\label{scaled expansion}
		u=U_0(1+\delta O(|X|)),\quad u_n=\frac{U_0}{r}(s+\delta O(|X|)),\quad \|g\|_{C^{2,\alpha}}\leq \delta,\; g\text{ is defined in \eqref{Fu smooth condistion}}.
	\end{equation}
	We notice that we only care about the regularity of $w$ near the free boundary $F(u)$. Therefore we can always scale $u$ and $w$. In the following Theorem, we only consider the scaled case.
	\begin{lemma}\label{uniform holder}
		Suppose $u$ is a solution to \eqref{extension of fracional laplacian} and satisfies the scaled expansion \eqref{scaled expansion} for some $\delta$ small enough. For \( w \) satisfying
		\[
		\left\{\begin{aligned}
			&|L_\beta (u_n w)| \leq |y|^\beta \delta\frac{U_0}{r} &\text{ in }& B_1 \cap\{u>0\}, \\
			&|w_{\nu}| \leq \delta &\text{ on }& F(u), \\
			&\|w\|_{L^\infty(B_1)} \leq 1.
		\end{aligned}\right.
		\]
		We have \( w \in C^{\eta} \) and \(\|w\|_{C^{\eta}(B_{1/2})} \leq C,\) for some small, universal \( \eta \).
	\end{lemma}
	\begin{proof}
		First, we note that for $\delta$ small enough, the equation is non-degenerate away from $F(u)$. Therefore the standard elliptic estimates guarantee the interior regularity of $w$.Our main task is to prove Hölder continuity near the set $\{u=0\}$. We handle this in two distinct regions. In step 1, we prove that $w$ is H\"{o}lder near $F(u)$. In step 2, we prove that $w$ is H\"{o}lder near $\{u=0\}\setminus F(u)$.
		
		\vspace{5pt}
		
		\textit{Step 1: Regularity on $F(u)$.}
		
		\vspace{5pt}
		
		In this step, we use the iteration argument to prove the H\"{o}lder regularity of $w$ near $F(u)$.  Under the setting \eqref{Fu smooth condistion}, we claim that
		\[\text{osc}_{B_{\delta_0}}w\leq 2-\delta_0,\]
		for some $\delta_0$ universal. We notice that after a translation $w_t(X)=w(X)+t$, we can always consider the case $\text{osc}_{B_1}w=2\|w\|_{L^\infty}$. Suppose the claim holds, we have that \[\text{osc}_{B_{\delta_0}}w\leq 2-\delta_0\leq (1-\frac{\delta_0}{2})2\|w\|_{L^\infty}=(1-\frac{\delta_0}{2})\text{osc}_{B_1}w .\] By iteration, we have that \(\|w\|_{C^{\eta_1}(\alpha)}\leq C\) for some $\eta_1\in(0,1)$.
		
		We prove the claim in 2 substeps. In substep 1, we construct barrier function $v$. In substep 2, we compare the barrier function $v$ with $u$. We will construct lower barrier function $v$ under the setting $w(\frac{1}{2}e_n)\geq 0$. For the case  $w(\frac{1}{2}e_n)\leq 0$, we can construct the upper barrier function similarly.
		
		\vspace{5pt}
		
		\textit{Substep 1: Barrier function $v$.}
		
		\vspace{5pt}
		
		We set \[v=-1+\delta_1(\frac{1}{4}+E(Q)+\frac{U_0}{u_n}(1+Mr)),\] where $Q(y')=-|y'|^2$, $M$ large and $\delta_1$ small to be determined later. By Lemma \ref{equation for u_nEQ}, we have
		\[L_\beta(u_nE(Q))\geq -C|y|^\beta\frac{U_0}{r}.\]
		
		By calculation we have
		\[\begin{aligned}
			\frac{1}{|y|^\beta}L_{\beta}(U_0(1+Mr))&=(1+Mr)L_\beta U_0+MU_0L_\beta r+2M(U_0)_r\\
			&=(1+Mr)s\kappa\frac{U_0}{r}+(1+\beta+\kappa)M\frac{U_0}{r}+2sM\frac{U_0}{r}\\
			&=((1+Mr)sO(\delta)+M(1+\beta+O(\delta))+2sM)\frac{U_0}{r}.
		\end{aligned}\]
		Here $\kappa$ is defined in \eqref{Laplacian t}. We choose $M$ big enough and $\delta$ small enough such that $L_{\beta}(U_0(1+Mr))\geq 2C|y|^\beta\frac{U_0}{r}$. By the two inequalities above, we have
		\begin{equation}\label{interior equation for v}
			L_\beta(u_nv)\geq C|y|^\beta \frac{U_0}{r}
		\end{equation}
		On $F(u)$ we have
		\begin{equation}\label{boundary condition of v}
			\partial_\nu(E(Q)+\frac{U_0}{u_n}(1+Mr))=\partial_\nu(\frac{r}{s}(1+O(X))(1+Mr))\geq c>0.
		\end{equation}
		
		\vspace{5pt}
		
		\textit{Substep 2: Comparison of $w$ and $v$ in $B_{3/4}\cap\{r<c\}$.}
		
		\vspace{5pt}
		
		In this substep, we compare $w$ and $v$ in $B_{3/4}\cap\{r<c\}$. We start from compare them on $\partial B_{3/4}\cap \{r<c\}$. For $c,\delta$ small enough, we have
		\[E(Q)\text{ is close enough to }(\frac{3}{4})^2,\quad \frac{U_0}{u_n}(1+Mr)=\frac{r(1+Mr)}{s+\delta O(|X|)}<\frac{2c}{s} \text{ on }\partial B_{3/4}\cap \{r<c\}.\]
		Therefore we have $v<-1\leq w\text{ on }\partial B_{3/4}\cap \{r<c\}$.
		
		Next we compare $w$ and $v$ on $B_{3/4}\cap\{r=c\}$.  We claim that that $1+w\geq c_0>0$ for some universal $c_0>0$ on $B_{3/4}\cap\{r=c\}$. By choosing $c,\delta_1$ small enough, we have $w\geq -1+c_0\geq v$ on $B_{3/4}\cap\{r=c\}$.
		
		We use Theorem \ref{interior harnack} Harnack and Theorem \ref{boundary harnack} boundary Harnack inequalities to prove the claim. Since $1+w(\frac{1}{2}e_n)\geq 1$ and $\delta$ is small, we apply Harnack inequality to compare $u_n(1+w)$ with $u_n$. It gives that 
		\[1+w=\frac{u_n(1+w)}{u_n}\geq c_1>0\text{ in }B_{7/8}\cap\{X:\text{dist}(X,\{u=0\})\geq \frac{c}{2}\}.\]
		Similarly, we apply Theorem \ref{boundary harnack} boundary Harnack inequality to compare $u_n(1+w)$ with $u_n$ under the setting $\delta$ small enough. The comparison is extended to the boundary $\{u=0\}$ and the claim follows.

		Finally, we show $w\geq v$ in $B_{3/4}\cap\{r<c\}$. For $\delta<<\delta_1$, \eqref{interior equation for v} indicates that the minimum of $u_n(w-v)$ cannot occur inside by maximum principle. The free boundary condition indicates that the minimum  of $u_n(w-v)$ cannot occur on $F(u)$. As analyzed above, the minimum  of $u_n(w-v)$ cannot occur on $\partial(B_{3/4}\cap \{r<c\})$. That shows $v$ is a lower barrier function.

		\vspace{5pt}
		
		\textit{Step 2: Regularity on $\{u=0\}\setminus F(u)$.}
		
		\vspace{5pt}
		
		For $X\in\{u=0\}\setminus F(u)$ and $2R=dist(X,F(u))\leq 1$, we prove the H\"{o}lder continuity of $w$ in $B_{R}(X)$. We prove it in 2 substeps. In substep 1, we set $H=u_n w$ and show the H\"{o}lder continuity of $\frac{H}{|y|^{1-\beta}}$. In substep 2, we show the H\"{o}lder continuity of $\frac{|y|^{1-\beta}}{u_n}$. The H\"{o}lder continuity of $w$ will be derived by the fact that $w=\frac{H}{|y|^{1-\beta}}\cdot\frac{|y|^{1-\beta}}{u_n}$.

		\vspace{5pt}
		
		\textit{Substep 1: H\"{o}lder continuity of $\frac{H}{|y|^{1-\beta}}$.}
		
		\vspace{5pt}
		
		In this substep, we show the H\"{o}lder continuity of $\frac{H}{|y|^{1-\beta}}$, $H=u_n w$. Set $\tilde{H}(X)=H(RX)$, we have \[|L_\beta \tilde{H}|\leq |y|^\beta R^2\frac{\tilde{U}_0}{r}.\] By calculus  we have \[|L_{2-\beta}(\frac{\tilde{H}}{|y|^{1-\beta}})|\leq |y|^{2-\beta}R^2\frac{\tilde{U}_0}{|y|^{1-\beta}r}\leq |y|^{2-\beta}CR\text{ in }B_{3/2}(\frac{X}{R}).\] Apply Theorem \ref{schauder estimate}, we have
		\[\|\frac{\tilde{H}}{|y|^{1-\beta}}\|_{C^1(B_{1}(\frac{X}{R}))}\leq C.\]
		Scaling it back, we have
		\[\|\frac{H}{|y|^{1-\beta}}\|_{C^1(B_{R}(X))}\leq\frac{C}{R}.\]

		\vspace{5pt}
		
		\textit{Substep 2: H\"{o}lder continuity of $\frac{|y|^{1-\beta}}{u_n}$.}
		
		\vspace{5pt}
		
		We use Theorem \ref{boundary harnack} boundary Harnack inequality to show the H\"{o}lder continuity of $\frac{|y|^{1-\beta}}{u_n}$. As we have\[L_\beta u_n=L_\beta (|y|^{1-\beta})=0,\;|u_n|\leq \frac{C}{R^{1-s}},\text{ in }B_{R}(X).\]Apply the boundary Harnack theorem, we have \[\|\frac{|y|^{1-\beta}}{u_n}\|_{C^{\eta_2}(B_{R}(X))}\leq \frac{C}{R^{1-s}}.\]  The desired result follows.
	\end{proof}
	
	With the uniform H\"{o}lder estimate, we can derive a compactness result by Arzela-Ascoli theorem. Here we give the result as below.
	
	\begin{lemma}\label{compactness}
		Assume a sequence $u_k$ of solutions  to \eqref{extension of fracional laplacian} satisfying \eqref{scaled expansion} for a sequence $\delta_k\to 0$. If a sequence $w_k$ satisfies
		\[\left\{\begin{aligned}
			&L_\beta(\partial_n(u_k)w_k)\leq \delta_k\frac{U_0}{r}\quad \text{in }B_1\setminus\pP_k,\\
			&|\partial_\nu w_k|\leq \delta_k\quad\text{on }\alpha,\\
			&\|w_k\|_{L^\infty(B_2)}\leq 1,
		\end{aligned}\right.\]
		then there is a subsequence of $w_{n_k}$ of $w_k$ such that $w_{n_k}$ converges uniformly on compact sets to a H\"{o}lder continuous function $w$. Furthermore, $w$ satisfies \eqref{linearized problem}.
	\end{lemma}
	\begin{proof}
		By Lemma \ref{uniform holder}, the existence and continuity of the limit $w$ is a direct result of Arzela-Ascoli Theorem. It suffices to prove that $w$ satisfies \eqref{linearized problem}.
		
		For the interior part, Theorem \ref{schauder estimate} indicates that $w$ is $C^{1,t}$ for any $t<1$. Condition \eqref{scaled expansion} indicates that $F(u)\to \Ll$, $(u_k)_n\to (U_0)_n$. Therefore $w$ satisfies $L_\beta(U_nw)=0$ in $B_1\setminus\pP$.
		
		For the boundary condition, we prove by contradiction. Suppose $w$ can be touched from below strictly at $0$ by a function $\varphi$ such that\[\varphi(X)=b_0-2a_1|y'-y'_1|^2+\frac{a_2}{2}r+Mr^2\quad\text{with }M>>2a_1,\]
		for some $b_0$ and $y_1'$. Since $w_k\to w$ uniformly, then we can touch $w_k$ from below by $\varphi_k=b_k-2a_1E(|y'-y'_1|^2)+s\frac{U_0}{u_n}(\frac{a_2}{2}+Mr)$.
		
		Now, as in Step 1 in the proof of Lemma \ref{uniform holder}, we can show that $\varphi_k$ is a subsolution to \eqref{linearized problem} for $k$ big enough. Therefore it cannout touch $w_k$ from below. This gives a contradiction.
	\end{proof}
	
	With all the preparation above, we give the proof of Theorem \ref{regularity improvement} below.
	\begin{proof}
		In the following proof, we will show that $w$ is $C^{k+2,\alpha}$ at $0$. The proof applies to all points on $F(u)$. We notice that $C^{k+2,\alpha}$ is a local regularity, therefore it suffices to consider a scaling of $w(R X)$ for $X\in B_{1/R}$. After a scaling, we only need to consider the case \eqref{scaled expansion} for some $\delta$ small.
		
		By definition, it suffices to find a $C^{k+2}$ function $W_{Q,P}$ to approximate $w$ at $0$. We determine $W_{Q,P}$ by approximating $w$ in sequence of balls $B_{\rho^m}$ inductively on $m$. Set\[W_{Q,P}=E(Q)+\frac{U_0}{u_n}P,\] with \begin{equation}\label{form of QP}
			\begin{aligned}
				&Q=Q(x')=q_\mu y^\mu,\;\text{$q_\mu=0$ for $\mu_n\neq0$ is a degree $k+2$ polynomial of $x'$},\\
				&P=P(x,r)=a_{\mu m}x^\mu r^m\;\text{is a degree $k+1$ polynomial of $(x,r)$.}
			\end{aligned}
		\end{equation}
		
		We claim that: if for a given $W_{Q,P}$ such that $\|w-W_{Q,P}\|_{L^\infty(B_{1})}\leq 1$ and $Q,P$ satisfy\begin{equation}\label{Neumann condition}
			\begin{aligned}
				&(i)&\;&\partial_\nu(\frac{U_0}{u_n}P) \text{ vanishes on \(F(u)\) at origin of order $k+1+\alpha$,}\\
				&(ii)&\;&\frac{L_\beta(u_nW_{Q,P})}{U_0/r} \text{ vanishes at origin of order $k+\alpha$,}
			\end{aligned}
		\end{equation}then there exists some $Q',P'$ satisfying \eqref{Neumann condition} such that $\|(Q'+rP')-(Q+rP)\|_{L^\infty(B_{1})}\leq C$, and $\|w-W_{Q',P'}\|_{L^\infty(B_{\rho})}\leq \rho^{k+2+s}$ for $\rho$ universal.
		
		We notice that the above claim implies that there exists $Q_m,P_m$ such that $W_{Q_m,P_m}$ approximating $w$ in $B_{\rho^m}$. We set $\bar{Q}=\lim_{m\to \infty}Q_m,\bar{P}=\lim_{m\to \infty}P_m$, then $|w-W_{\bar{Q},\bar{P}}|\leq C|X|^{k+2+s}$ by standard argument. That is the desired result.
		
		Now it suffices to prove the claim. To begin with, we start by investigating \eqref{Neumann condition} $(i)$. It indicates that $\partial_\nu W_{Q,P}=O(|X|^{k+1+s})$. By Lemma \ref{expansion theorem}, we have $\frac{U_0}{u_n}=r(s+P^*)$. Therefore it suffices to require $P^*P$ to vanish of order $k+1$ at origin. Consider the coefficients of $x^{\mu'}$, we have
		\begin{equation}\label{amu1 determine}
			a_{(\sigma',0),0}=\hat{c}^\mu_{\sigma'}a_{\mu,0},\;|\hat{c}^\mu_{\sigma'}|<\delta,\;\hat{c}^\mu_{\sigma'}\text{ only if }|\mu|<|\sigma'|.
		\end{equation}
		
		Next we analyze \eqref{Neumann condition} $(ii)$. We calculate as below \[\begin{aligned}
			L_\beta(u_nW_{Q,P})	=|y|^\beta\frac{U_0}{r}(A_{\sigma l}x^\sigma r^l+O(|X|^{k+s}))+L_\beta(U_0P),
		\end{aligned}\]
		where $A_{\sigma l}\text{ is given in Lemma \ref{equation for u_nEQ}}.$ For the second term above, we calculate following \cite{Jhaveri2017} (4.1) to derive that \[\begin{aligned}
			L_\beta(U_0x^\mu r^m)&=|y|^\beta\frac{U_0}{r}(-(dm+sr)\kappa x^\mu r^{m-1}+m(m+1)x^\mu r^{m-1}\\
			&+2r^{m-1}(dm+sr)\nu\cdot \nabla_x x^\mu+\mu_i(\mu_i-1)x^{\mu-2\bar{i}}r^{m+1})\\
			&=|y|^\beta\frac{U_0}{r}(m(m+1_2\mu_n)x^\mu r^{m-1}+2s\mu_nx^{\mu-\bar{n}}r^m\\
			&\mu_i(\mu_i-1)x^{\mu-2\bar{i}}r^{m+1}+c_{\sigma l}^{\mu m}x^\sigma r^l+h^{\mu m}(x,r)).\\
		\end{aligned}\] Here $c_{\sigma l}^{\mu m}\neq0\text{ only if }|\mu|+m\leq |\sigma|+l\leq k,$ $|c_{\sigma l}^{\mu m}|<\delta$, and $h^{\mu m}(x,r)=r^mh_m^\mu(x)+mr^{m-1}h_{m-1}^\mu(x)$ and $h^\mu_m,h^\mu_{m-1}\in C^{k,\alpha}(B_1^n)$ has vanishing derives of order $k-m$ and $k-m+1$ at the origin. That indicates
		\[L_\beta(U_0P)=|y|^\beta\frac{U_0}{r}(B_{\sigma l}x^\sigma r^l+\sum_{m=0}^kr^mh_m(x)),\]
		where $B_{\sigma l}$,$h_m$ satisfy\[\begin{aligned}
			&B_{\sigma l}=(l+1)(l+2+2\sigma_n)a_{\sigma,l+1}+2s(\sigma_n+1)a_{\sigma+\bar{n},l}+(\sigma_i+1)(\sigma_i+2)a_{\sigma+2\bar{i},l-1}+c_{\sigma l}^{\mu m}a_{\mu m},\\
			&\|h_m\|_{C^{k,\alpha}(B_1^n)}\leq C,\; h_m\text{ vanish of order }k-m\text{ at origin.}
		\end{aligned}\]
		
		In conclusion, \eqref{Neumann condition} $(ii)$ is equal to $A_{\sigma l}+B_{\sigma l}=0$. Now we show that $q_\mu$ and $a_{\mu 0}$ in \eqref{form of QP} will uniquely determine $Q$, $P$ satisfying \eqref{Neumann condition}. We prove by induction. Suppose for $l\in \Nn$, $a_{\mu l}$ are determined. Then we have \begin{equation}\label{amul induction}
			-A_{\sigma l}-2s(\sigma_n+1)a_{\sigma+\bar{n},l}+(\sigma_i+1)(\sigma_i+2)a_{\sigma+2\bar{i},l-1}+c_{\sigma l}^{\mu m}a_{\mu m}=(l+1)(l+2+2\sigma_n)a_{\sigma,l+1}.
		\end{equation} As $A_{\sigma l}$ is uniquely determined by $q_\mu$ and $c_{\sigma l}^{\mu m}\neq0\text{ only if }|\mu|+m\leq |\sigma|+l\leq k,$, we are done. Furthermore, by \eqref{amu1 determine} we only need $q_\mu$ and $a_{\mu,0}$, $\mu\neq 0$ to determine $Q,P$ uniquely.
		
		We  prove the claim to end this proof. Set
		\[w=W_{Q,P}+\tilde{w},\]
		we have
		\[\left\{\begin{aligned}
			&|L_\beta(u_n\tilde{w})|\leq C\delta|y|^\beta\frac{U_0}{r},\\
			&|\tilde{w}_\nu|\leq \delta\text{ on }\alpha.
		\end{aligned}\right.\]
		From Lemma \ref{compactness} and Theorem \ref{quadratic expansion}, for $\delta$ small enough there exists $\tilde{Q}$ and $\tilde{P}$ such that
		\[|\tilde{w}-\tilde{Q}-r\tilde{P}|\leq \frac{1}{4}\rho^{k+2+\alpha}+C\rho^{k+3}\leq\frac{1}{2}\rho^{k+2+s}\quad\text{in }B_\rho.\]
		Furthermore, $\tilde{Q}$ and $\tilde{P}$ satisfy \eqref{Neumann condition} for $F(u)=\{(x',x_n,y):x_n=y=0\}$. We modify $\tilde{Q},\tilde{P}$ to construct $Q'$, $P'$ in the claim.
		
		By the argument above, we use $q_\mu$ of $\tilde{Q}$ and $\tilde{a}_{\mu 0}$ of $\tilde{P}$ as coefficients. The generated $Q'$,$P'$ satisfying \eqref{Neumann condition} is unique. Furthermore, since $\tilde{Q}$ and $\tilde{P}$ satisfy the equations \eqref{amu1 determine} and \eqref{amul induction}  with $c_{\sigma l}^{\mu  m}=\hat{c}^\mu_{\sigma'}=0$, we have $\|\tilde{Q}+r\tilde{P}-\bar{Q}-r\bar{P}\|\leq C\delta$. Finally, by the fact that $E(\bar{Q})$ is close to $\bar{Q}$ for $\rho$ small, we have
		\[|\tilde{w}-W_{\bar{Q},\bar{P}}|<\rho^{k+2}\quad\text{in }B_\rho.\]
		That gives the desired claim.
	\end{proof}
	
	Finally, we prove Theorem \ref{C infty of free boundary} to end this section.
	\begin{proof}
		Since $u(x',g(x'),0)=0$, simple calculus gives that
		\[u_i+u_ng_i=0,\;1\leq i\leq n-1.\] Therefore it suffices to show that $w=\frac{u_i}{u_n}$ has the same regularity as $g$. We prove in 2 steps. In step 1, we apply Theorem \ref{expansion theorem} to show that $\frac{u_i}{u_n}$ satisfies the equation \eqref{equation for gi}. In step 2, we apply Theorem \ref{regularity improvement} to show that $w$ has the same regularity as $g$.
		
		\vspace{5pt}
		
		\textit{Step 1: Equation for $w$.}
		
		\vspace{5pt}
		
		For the interior equation, we have
		\[L_\beta(u_n\frac{u_i}{u_n})=L_\beta(u_i)=0,\]
		it suffices to prove the Neumann boundary condition.
		
		From Theorem \ref{expansion theorem} and $\nabla g(0)=0$, there exists a  polynomial $P=a_0+a_nx_n	+a_{n+1}r$, $a_n\neq 0$ such that $\frac{u_i}{sU_0/r}=\partial_iP+O(r^{1+\alpha})$, $1\leq i\leq n$. Therefore we have
		\[\frac{u_i}{u_n}=\frac{O(r^{1+\alpha})}{a_n+O(r^{1+\alpha})}.\]
		That indicates the Neumann condition.

		\vspace{5pt}
		
		\textit{Step 2: Improvement of regularity.}
		
		\vspace{5pt}
		
		From Theorem \ref{regularity improvement}, $w$ has the same regularity as $g$. Suppose $g$ is $C^{k+2,\alpha}$, then $g_i$ is also $C^{k+2,\alpha}$. That implies $g$ is $C^{k+3,\alpha}$. The induction shows that $g$ is smooth.
	\end{proof}
	
	\section{Expansion Theorem for Linearized Problem}\label{expansion theorem section}
	In this section, we prove the expansion Theorem \ref{quadratic expansion}. It gives an infinite order expansion to the linearized problem \eqref{linearized problem}. For $\beta=0$, the equation $\Delta (U_nw)=0$ in $\Rr^n\setminus\pP$ is easier to understand. We can consider the slice of the space $\{x'=0\}$ and map the 2D slit domain to the upper half plane with a holomorphic map. For general $\beta$, we will use boundary Harnack inequality in slit domain to handle it. The idea is inspired by \cite{DeSilva2012b,Jhaveri2017}.
	
	For $k=1$, the theorem was proven in \cite{Silva2014} by D. De Silva, O. Savin, and Y. Sire. Based on that, we will prove by induction on $k$. Our proof is inspired by the paper \cite{DeSilva2012b} by D. De Silva and O. Savin. Furthermore, the regularity results in \cite{Jhaveri2017} by Y. Jhaveri and R. Neumayer and \cite{Sire2021a} by Y. Sire, S. Terracini, and S. Vita play important roles in our proof. To begin with, we recall the boundary Harnack theorem in slit domain. It is a special case of Theorem \ref{slit domain boundary harnack}.
	\begin{theorem}\label{boundary harnack for half space}(\cite[Proposition 3.3]{Jhaveri2017})
		Assume that $u\in C(B_1)$ is a solution to \eqref{extension of fracional laplacian}, even with respect to $\{y=0\}$, and $F(u)=\Ll$. For any $k\in\Nn$, there exists a polynomial $P(x,r)$ of degree $k$ with $\|P\|\leq C$ such that\[\left\{\begin{aligned}
			&L_\beta (UP)=0,\quad&\text{ in }&B_1\setminus \pP,\\
			&|u-UP|\leq CU|X|^{k+1},&\text{ in }&B_1\setminus\pP,
		\end{aligned}\right.\]
		for some constant $C=C(\beta,n,k)$.
	\end{theorem}
	\begin{remark}\label{remark boundary harnack for half space}
		Notice that for $k\in\Nn$, it suffices to have $|L_\beta u|\leq C|y|^\beta |X|^{k-1}U$ to derive the same result. To prove this, we calculate as below.
		\[\begin{aligned}
			L_\beta (|X|^{k+1}U)=|y|^\beta(\Delta+\frac{\beta}{y})(|X|^{k+1}U)=(k+1)(n+k+1)|y|^\beta |X|^{k-1}U.
		\end{aligned}\]
		By the calculation above, for $u$ satisfying $|L_\beta u|\leq C|y|^\beta |X|^{k-1}U$, it suffices to consider $u\pm C|X|^{k+1}U$ as barrier functions. In detail, we have $L_\beta(u+C|X|^{k+1}U)\geq 0$. Then we consider $u_1$ satifying
		\[\left\{\begin{aligned}
			&L_\beta u_1=0,\quad&\text{in }&B_1\setminus \pP\\
			&u_1=u+Cr^{k+1}U,\quad&\text{on }&\partial(B_1\setminus \pP).
		\end{aligned}\right.\]
		By Theorem \ref{interior harnack} Harnack inequality, we have $|u_1-UP_1|\leq C|X|^{k+1}U$. By maximum principle, we have
		\[u\leq u_1\leq UP_1+C|X|^{k+1}U.\]
		Similarly, we can derive that
		\[u\geq UP_2-C|X|^{k+1}U.\]
		We can simply derive that
		\[|P_1-P_2|\leq C|X|^{k+1}.\]
		As $P_1$ and $P_2$ are polynomials of degree $k$, we can say they are exactly the same.
	\end{remark}
	The proof of Theorem \ref{quadratic expansion} is based on the fact that for the solution $h$ to \eqref{linearized problem}, $D^k_{x'}h$ has the same regularity as $h$. To begin with, we present some known results first.
	
	\begin{lemma}\label{derivative regularity}(\cite[Theorem 6.1 and Lemma 6.3]{DeSilva2014})
		Given a boundary data $\bar{h}\in C(\partial B_1)$ even with respect to $\{y=0\}$, $|\bar{h}|\leq 1$. There exists a unique classical solution $h$ to \eqref{linearized problem} such that $h\in C(\bar{B}_1)$ even with respect to $\{y=0\}$, $h=\bar{h}$ on $\partial B_1$.
		
		Furthermore, $D^k_{x'}h$ is also a solution to \eqref{linearized problem} for any $k\in \Nn^n$. There exists $\alpha\in (0,1)$ universal such that
		\[\begin{aligned}
			|h(X)-h(0)-\xi_1\cdot x'|&\leq C(n,\beta)(|x'|^2+r^{1+\alpha}),\\
			[D^k_{x'}h]_{C^\alpha(B_{1/2})}&\leq C(n,\beta,k)\|h\|_{L^\infty(B_1)}.
		\end{aligned}\]
		Here $|\xi_1|\leq C(n,\beta)$.
	\end{lemma}
	\begin{remark}\label{derivative regularity remark}
		Lemma \ref{derivative regularity} indicates the H\"{o}lder regularity of $D^k_{x'}h$. In fact, $D^k_{x'}h$ shares the same regularity with $h$. By Lemma \ref{derivative regularity} we have that $D^k_{x'}h$ is also a solution to \eqref{linearized problem} with uniform bound. We apply the regularity estimate of $h$ to $D^k_{x'}h$ and have the result.
	\end{remark}
	
	Now we give the proof of Theorem \ref{quadratic expansion}. Our proof is inspired by the paper \cite{DeSilva2014}.
	
	\begin{proof}
		We will prove the desired results in 2 steps. In step 1, we show that for $\alpha$ in Lemma \ref{derivative regularity} and $\forall k\in \Nn$, there exists a polynomial $P_k(x,r)$ of order $k$ such that
		\begin{equation}\label{holder approximation}
			|h-P_k|\leq C_o|X|^{k+\alpha}\text{ in }B_{1/2}.
		\end{equation}
		In step 2, we consider the properties of $P_k$. We show that $P_k$ is bounded and $L_\beta(U_nP_k)=0$ for all $k$. Then by the bound of $P_{k+1}$, we have that the $\alpha$ in step 1 can be raised to 1.
		
		\vspace{5pt}
		
		\textit{Step 1: Proof of \eqref{holder approximation}:}
		
		\vspace{5pt}
		
		We prove by induction on k in this step. The case $k=1$ was proven in \cite[Theorem 6.1]{DeSilva2014}. We assume \eqref{holder approximation} for $k$ and prove the $k+1$ case. We will prove in 2 substeps. In substep 1, we show that there exists a polynomial approximating $h$ in the space $\{x'=0\}$. In substep 2, we extend the result to the whole space.
		
		\vspace{5pt}
		
		\textit{Substep 1: Polynomial appromating $h$ in $\{x'=0\}$.}
		
		\vspace{5pt}
		
		In this step, we consider $h$ in $\{x'=0\}$. In this 2D space, we will use coordinates $t$ to denote original $x_n$. $y,r$ remain the same as in $\Rr^{n+1}$. By definition, we have
		\[L_\beta(U_nh)=0\quad\text{in }B_1\setminus \pP.\]
		Since $U_n$ and $y$ are independent of $x'$, we extract the terms of derivatives in $x'$ direction to have
		\[\text{div}_{x_n,z}(|y|^\beta\nabla(U_nh))=-|y|^\beta U_n\Delta_{x'}h.\]
		
		As indicated in Remark \ref{derivative regularity remark}, we can apply the regularity result of $h$ to $f=\Delta_{x'}h$ to derive
		\[|f-Q(x,r)|\leq C_o|X|^{k+\alpha},\]
		where $C_o$ is a universal constant and  $\|Q\|$ is bounded by universal constants and $k$. Here $Q$ is an order $k$ polynomial of $t,r$. In $\{x'=0\}$, $\Delta_{x'}h$ is a function of $t$ and $r$. We use  $f$ to denote $\Delta_{x'}h$. We investigate the 2-dimensional problem in $(t,y)$-variables:
		\begin{equation}\label{2D regularity problem}
			\text{div}_{t,y}(|y|^\beta\nabla(U_th))=|y|^\beta U_tf\quad \text{in }B_{1/2}\setminus\pP.
		\end{equation}
		The regularity analysis of \eqref{2D regularity problem} derives regularity results on sets $\{x'=0\}$. We consider the solution $H\in C(B_{1/2})$ to the following problem
		\[\left\{\begin{aligned}
			&\text{div}_{t,y}(|y|^\beta\nabla H)=|y|^\beta U_tf\quad&\text{in }&B_{1/2}\setminus\pP,\\
			&H=U_th\quad&\text{on }&\partial B_{1/2},\\
			&H=0\quad&\text{on }&B_{1/2}\cap\pP.
		\end{aligned}\right.\]We claim that $H=U_th$ in $B_{1/2}$. By the equality, it suffices to analyze the regularity of $H$. 
		By Remark \ref{derivative regularity remark}, we have that for some universally bounded polynomial $P(t,r)$ of order $k+1$, \[|\frac{H}{U}-P|=|s\frac{h}{r}-P|\leq C(\beta,n,k)r^{k+1+\alpha}.\]
		This gives the desired result in the plane $\{x'=0\}$. 
		
		We prove the claim to end this substep. As we have
		\[\left\{\begin{aligned}
			&\text{div}_{t,y}(|y|^\beta\nabla (H-U_th))=0\quad&\text{in }&B_{1/2}\setminus\{t\leq 0,y=0\},\\
			&H-U_th=0\quad&\text{on }&\partial B_{1/2}\cup(B_{1/2}\cap\{t< 0,y=0\}),
		\end{aligned}\right.\]
		it suffices to analyze $H-U_nh$ near $\{t=0,y=0\}$ to  apply maximum principle.
		
		By Theorem \ref{slit domain boundary harnack} Harnack inequality in slit domain,  we have that $\frac{H}{U}\in C^\alpha(B_{1/2})$. In other words, $\exists a\in \RR$ such that $|H-aU|\leq Cr^\alpha U$ in $B_{1/2}\setminus\{t\leq 0,y=0\}$. Therefore, $$\lim_{r\to 0} \frac{H-U_th}{U_t}=\lim_{r\to 0} \frac{H}{U_t}=\lim_{r\to 0} \frac{H/U}{s/r}=0.$$ By maximum principle, we have $-\epsilon U_t\leq H-U_th\leq \epsilon U_t$ for any $\epsilon>0$ small. Letting $\epsilon\to 0$, we have the desired result.

		\vspace{5pt}
		
		\textit{Substep 2: Polynomial appromating $h$ in the whole space.}
		
		\vspace{5pt}
		
		We apply the result of substep 1 to each slice $\{x'=c\}$. It gives
		\[|h-\sum_{m,l\in\Nn}^{l+m\leq k+1}b_{\mu l}^h(x')x^m r^l|\leq C(\beta,n,k)r^{k+1+\alpha}.\]
		It suffices to show that all $b_{m l}^h(x')$ are $x'$ smooth functions and apply Taylor expansion to them at $x'=0$. To show this, we consider $h^\tau(X)=\frac{h(X+\tau)-h(X)}{|\tau|}$ for $\tau$ in the $x'$ direction. $h^\tau$ is a solution to \eqref{linearized problem}, and $b_{m l}^{h^\tau}(x')=\frac{b_{m l}^h(x'+\tau)-b_{m l}^h(x')}{|\tau|}$ by definition. Pass $|\tau|\to 0$. According to Remark \ref{derivative regularity remark}, $D^\tau h$ is uniformly bounded. Therefore we have $\partial_\tau b_{m l}^h=b_{m l}^{\partial_\tau h}$ being bounded by universal constants and $k$. This gives the desired result.

		\vspace{5pt}
		
		\textit{Step 2: Properties of $P_k$.}

		\vspace{5pt}
		
		To begin with, we show that $P_k$ is bounded by constant by $C(n,\beta,k)$. We notice that it suffices to prove the bound for $b_{ml}$ above. It is a direct result of Remark \ref{derivative regularity remark}.
		
		Next we prove the $L_\beta P=0$ by induction. Use $P_k$ to denote the approximating polynomial of order $k$. For $k=1$, the equality holds naturally. Assume the equality holds for $k$, we prove it for $k+1$.
		
		We set $v=u-U_0P_k$. It satisfies that $L_\beta v=0$. By the expansion proven above, $v=U_0(P_{k+1}-P_k+o|X|^{k+1})$. We consider\[v_\lambda(X)=\frac{v(\lambda X)}{\lambda^{s+k+1}},\]and pass  $\lambda$ to $0$. It converges to $U_0(P_{k+1}-P_k)$ uniformly. Therefore we have
		\[L_\beta(U_0(P_{k+1}))=L_\beta(U_0(P_{k+1}-P_k)+U_0P_k)=0.\]
		
		Finally, we notice that the $\alpha$ in Step 1 can be replaced by $1$. All we need to do is to consider the $k+1$ case, the $k+1+\alpha$ regularity is stronger than $k+1$ regularity.
	\end{proof}
	\section{Appendix}\label{appendix}
	In this Appendix, we gather some known results about flat boundaries, Harnack inequalities,  and regularity theories. To begin with, we recall the Harnack type inequalities for the operator $L_\beta$. To begin with, we recall Harnack inequality. It will be used to estimate the solution $u$.  
	
	\begin{theorem}\label{interior harnack}(Harnack inequality, \cite[Theorem 2.4]{tangentball} or \cite[Theorem 2.13]{DeSilva2014})
		
		Let $u \geq 0$ be a solution of  
		\[  
		L_\beta u = 0 \quad \text{in} \quad B_1 \subset \mathbb{R}^n.  
		\]  
		Then,  
		\[  
		\sup_{B_{1/2}} u \leq C \inf_{B_{1/2}} u,  
		\]  
		for some constant $C$ depending only on $n$ and $\beta$.  
	\end{theorem} 
	Next we state the boundary Harnack inequality. Similar to Harnack inequality, boundary Harnack inequality will be used to analyze the behavior of $u$ near $\{u=0\}$.
	\begin{theorem}\label{boundary harnack}(Boundary Harnack inequlity, \cite[Theorem 2.4 (ii)]{tangentball})
		
		Assume that $u\geq 0$, $v$ satisfy
		\[\left\{\begin{aligned}
			&L_\beta u=L_\beta v=0 &\text{ in }& B^n_1(x_0)\times(y_0-\frac{1}{2},y_0+\frac{1}{2}),\\
			&u=v=0&\text{ on }&\Sigma=Q_{1/2}^n(x_0)\times\{y_0+1/2\},\\
			&u(x_0,y_0)=v(x_0,y_0),
		\end{aligned}\right. \]
		then $\|\frac{v}{u}\|_{C^\eta}\leq C$ for some universal $\eta\in(0,1)$ and $C$.
	\end{theorem}
	Before we  move on, we point out that the inequalites above can apply to a more general operator. For operator in the form div$(\varphi\nabla\cdot)$, the inequalites apply if for any ball $B$ we have $\sup(\frac{1}{|B|}\int_B \varphi)(\frac{1}{|B|}\int_B \varphi^{-1})\leq c$. We refer to \cite{Fabes1982,Fabes1982no2} for details. Finally, we introduce Boundary Harnack in slit domains. It is used to analyze the behavior of $u$ near the free booundary. The pioneering paper \cite{Silva2014} first provided the inequality for $\beta=0$. The paper \cite{Jhaveri2017} proved the inequlity for $-1<\beta<1$. A more recent paper \cite{Kassmann2024} contains some more general Harnack inequalities.
	\begin{theorem}\label{slit domain boundary harnack} (Boundary Harnack in slit domain, \cite[Theorem 1.3]{Jhaveri2017})
		
		Suppose that $\Omega\subset\Rr^n$ satisfy $\|\partial\Omega\|_{C^{k,\alpha}}\leq 1$, $0\in \partial\Omega$, and the normal unit vector of $\partial\Omega$ at $0$ is $e_n$, the unit vector in $x_n$ direction. Furthermore, \( u \) and \( U > 0 \) satisfy\[\left\{\begin{aligned}
			&L_\beta u=L_\beta v=0\quad\text{ in }B_1 \setminus(\Omega\times\{0\}),\\
			&u,v\text{ vanishe continuously on the slit }B_1\cap (\Omega\times\{0\}),\\
			&u,v \text{ are even in $y$},\\
			&\|u\|_{L^\infty(B_1)} \leq 1 \quad \text{and} \quad v\Big(\frac{\nu(0)}{2}\Big) = 1.
		\end{aligned}\right.\]
		Then  
		\[
		\left\| \frac{u}{v} \right\|_{C^{k,\alpha}_{x,r}(\alpha \cap B_{1/2})} \leq C,  
		\]  
		for some \( C = C(\beta, n, k, \alpha) > 0 \).
	\end{theorem}  
	Finally, we introduce a sequence of regularity theorems to finish this Appendix. To begin with, we state a regularity theorem for the operator $L_\beta$. It is used to analyze the behavior of the modified function near $\{y=0\}$.
	\begin{proposition}\label{interior regularity for nabla_xu and uy}(\cite[Proposition 2.3]{Jhaveri2017})
		Let $u$ be a solution to \[L_\beta u=|y|^\beta f\text{ in }B_\lambda,\;f\in L^\infty(B_\lambda).\]
		Then we have
		\[\|\nabla_x u\|_{L^\infty(B_{\lambda/4})}\leq \frac{C}{\lambda}\|u\|_{L^\infty(B_\lambda)}+C\|f\|_{L^\infty(B_\lambda)}.\]
		Furthermore, if $\bar{f}:=|y|^\beta\partial_yf\in L^{\infty}(B_\lambda)$, then
		\[\||y|^\beta\partial_y u\|_{L^\infty(B_{\lambda/4})}\leq C\lambda^{\beta-1}\|u\|_{L^\infty(B_\lambda)}+C\lambda^\beta\|f\|_{L^\infty(B_\lambda)}+C\lambda^2\|\bar{f}\|_{L^\infty(B_\lambda)}.\]
	\end{proposition}
	For the interior part, we will need the following Theorem. It is used to analyze the difference between the approximating function and paraboloid with the solution.
	\begin{theorem}\label{schauder estimate}(\cite[Theorem 1.2]{Sire2021a})
		Suppose $u\in H^1{(|y|^\beta B_1)}$ and $f\in L^\infty(B_1)$ satisfy
		\begin{equation}\label{disturbed Lbeta equation}
			\left\{\begin{aligned}
				&L_\beta u=|y|^\beta f\quad &\text{in }&B_1(X)\subset\RR^{n+1},\\
				&|y|^\beta u_y=0\quad&\text{on }&B_1(X)\cap\{y=0\}
			\end{aligned}\right.
		\end{equation}
		in the weak sense. For any $t\in (0,1)$, we have:
		\begin{equation}\label{schauder estimate inequality}
			\|u\|_{C^{1,t}(B_{1/2}(X))}\leq C(\|u\|_{L^\infty(B_1(X))}+\|f\|_{L^\infty(B_1(X))}),
		\end{equation}
		where $C$ depends only on $n,t,\beta$.
	\end{theorem}
	
	For $f\equiv 0$, $t$ can be chosen to be exactly $1$. The estimate was proven in \cite{Caffarelli2007a}, \cite{Caffarelli2007} and \cite{Fabes1982no2}. We state it to end this paper.
	
	\begin{theorem}\label{regularity of harmonic}(\cite[Proposition 3.5]{Caffarelli2016})
		Let \( u \in H^1(B_1, y^\beta \, dX) \) be a weak solution to
		\begin{equation}\label{harmonic solution}
			\left\{\begin{aligned}
				&L_\beta u = 0 \quad &\text{in }& B_{3/4},\\
				&|y|^\beta u_y\Big|_{y=0}=0\quad&\text{on }&B_{3/4}^n.
			\end{aligned}\right.
		\end{equation}
		Then we have the following results.
		\begin{enumerate}
			\item For each \( B_r(x_0) \subset B_1 \),
			\[
			\max_{B_{r/2}(x_0)} |u| \leq M \left( \frac{1}{r^{n+1+\beta}} \int_{B_r(x_0)} y^a |u|^2 \, dX \right)^{1/2},
			\]
			where \( M \) depends only on \( n \) and \( s \).
			
			\item We have
			\[
			\sup_{x \in B_{1/2}} |u_y(x, y)| \leq C y,
			\]
			where \( C \) depends only on \( n \) and \( s \).
			
			\item For each integer \( k \geq 0 \) and each \( B_r(x_0) \subset B_1 \),
			\[
			\sup_{B_{r/2}(x_0)} |D_x^k u| \leq \frac{C}{r^k} \text{osc}_{B_r(x_0)} u,
			\]
			where \( C \) depends only on \( n \), \( k \), and \( s \).
			
		\end{enumerate}
	\end{theorem}
\bibliographystyle{alpha} 
\bibliography{bibtex}
\end{document}